\documentclass[11pt,reqno]{amsart}
\usepackage{graphicx}
\usepackage{float}
\usepackage[caption = false]{subfig}
\usepackage{framed}
\usepackage{enumerate}
\usepackage{mathtools}
\usepackage{breqn}
\usepackage{array, geometry, graphicx}
\usepackage{amsmath,amsfonts,paralist,amssymb,amsthm,mathrsfs}
\usepackage{pdfpages}
\newtheorem{theorem}{Theorem}[section]
\newtheorem{corollary}[theorem]{Corollary}
\newtheorem{lemma}[theorem]{Lemma}

\theoremstyle{definition}
\newtheorem{definition}[theorem]{Definition}
\theoremstyle{remark}
\newtheorem{remark}[theorem]{Remark}
\numberwithin{equation}{section}
\DeclareMathOperator{\RE}{Re}

\begin{document}
	
\title[Non-Univalent functions and a parabolic region]{On a Class of Non-Univalent functions Associated with a Parabolic Region}
	\thanks{The first author is supported by Delhi Technological University, New Delhi}
	\author[Mridula Mundalia]{Mridula Mundalia}
	\address{Department of Applied Mathematics, Delhi Technological University, Delhi--110042, India}
	\email{mridulamundalia@yahoo.co.in}
	\author{S. Sivaprasad Kumar}
	\address{Department of Applied Mathematics, Delhi Technological University, Delhi--110042, India}
	\email{spkumar@dce.ac.in}

	
\begin{abstract}
In the present investigation, we introduce and study the geometric properties of a class of analytic functions, associated with a parabolic region majorly lying in the left-half plane. Further we establish radius and majorization results for the class under study with pictorial illustrations of some of its special cases. Also we derive some sufficient conditions for the class under consideration.
\end{abstract}
	
\maketitle
	
\section{Introduction}
Let $\mathcal{A}$ be the class of analytic functions $f(z)$ defined on the open unit disc $\mathbb{D}=\left\{z:|z|<1\right\}$ with the normalization $f(0)=0$ and $f'(0)=1.$ Assume $\mathcal{S}\subset\mathcal{A}$ to be the class of univalent functions. Let $f(z)$ and $g(z)$ be two analytic functions, then $f(z)$ is said to be subordinate to $g(z),$ symbolically $f\prec g,$ if there exist a Schwarz function $w(z)$ in $\mathbb{D}$ with $w(0)=0,$ such that $f(z)=g(w(z)).$ Additionally, if $g(z)$ is univalent in $\mathbb{D},$ then $f\prec g$ if and only if $f(\mathbb{D}_{r})\subset g(\mathbb{D}_{r}),$ where $\mathbb{D}_{r}=\{z:|z|<r<1\}.$ Recall that a function $f\in\mathcal{A}$ is starlike if $f(\mathbb{D})$ is starlike with respect to 0. Analytically, a function $f\in \mathcal{A}$ is starlike if  
\[\frac{zf'(z)}{f(z)}\prec \frac{1+z}{1-z}.\] 
We denote this class by $\mathcal{S}^{*}.$ Ma and Minda \cite{Ma & Minda} gave a unified representation for various subclasses of starlike functions by replacing the superordinate function $(1+z)/(1-z)$ with a more general function $\phi(z),$ and the corresponding class is denoted by $\mathcal{S}^{*}(\phi).$ Here $\phi(z)$ is chosen such that it is univalent, starlike with respect to $\phi(0)=1$ and $\operatorname{Re}\phi(z)>0$ with  $\phi'(0)>0,$ also  $\phi(\mathbb{D})$ is symmetric about real axis. Several Ma-Minda subclasses have been studied previously (See Table \ref{EQNTable2}). 
In contrast to $\mathcal{S}^{*}(\phi),$  Uralegaddi \cite{Uralegaddi(1994)} introduced and studied the class  \[\mathcal{M}(\beta)=\{f\in\mathcal{A}:\operatorname{Re}(zf'(z)/f(z))<\beta, \text{ }\beta>1\}.\] Note that $\mathcal{M}(\beta)\nsubseteq\mathcal{S}^{*}$ and also contains non-univalent functions. In 2006, Ravichandran et al. \cite{Ravi n Silverman(2006)} computed the radius of starlikeness for the class $\mathcal{M}(\beta)$. Motivated by the above class, Kumar et al. \cite{Gangania n Kumar(2021)Trans} 
made a systematic study of the class $\mathcal{F}(\psi)$ containing non-univalent functions, given by
\[\mathcal{F}(\psi):=\left\{f\in\mathcal{A}:\frac{zf'(z)}{f(z)}-1\prec \psi(z)\right\},\]
where $\psi(z)$ is an analytic univalent function such that $\psi(\mathbb{D})$ is starlike with respect to $0$ and $\psi(0)=0.$  In general $\mathcal{F}(\psi)\nsubseteq\mathcal{S}^{*}(\phi).$ Further if $\phi(z)=1+\psi(z) \prec (1+z)/(1-z),$ then $\mathcal{F}(\psi)$ reduces to $\mathcal{S}^{*}(1+\psi).$ Here below we give a list of functions $\psi_{i}(z),$ $i=1,\ldots,4,$ which are considered in context of the above class $\mathcal{F}(\psi),$
\begin{align}\label{EQN420}
\psi_{i}(z)=
\left\{    
    \begin{array}{ll}
     \gamma z(1+\eta z)^{-2} & ,i=1 \\
     z(1-\alpha z^{2})^{-1} & ,i=2 \\
     z(1-z)^{-1}(1+\beta z)^{-1}  & ,i=3 \\
     (A-B)^{-1}\log((1+Az)/(1+Bz)) & ,i=4,
    \end{array}
\right.
\end{align}
where $A = \alpha e^{i\tau},B=\alpha e^{-i\tau}$ with $\tau\in(0,\pi/2],$ and $\alpha,\beta,\eta\in(0,1],\gamma>0$  (see \cite{Kargar(2019),Gangania n Kumar(2021)Trans,Kumar n Yadav(2022),Masih n Ebadian(2019)}). Kumar and Gangania in \cite{Gangania n Kumar(2021)Trans}, 
introduced the class $\mathcal{S}_{\gamma}(\eta)=\mathcal{F}(\psi_{1})$ and obtained the radius of starlikeness. Cho et al. \cite{Cho n Ravi(2018)} dealt with certain sharp radius problems for the class $\mathcal{BS}(\alpha)=\mathcal{F}(\psi_{2}).$ Infact Masih et al. \cite{Masih n Ebadian(2019)} studied the class $\mathcal{S}_{cs}(\beta)=\mathcal{F}(\psi_{3}),$ where $0 \leq \beta < 1,$ discussed the growth theorem and established sharp estimates of logarithmic coefficients for $0\leq \beta \leq 1/2.$ Further for $1/2<\beta\leq 1,$ the class $\mathcal{S}_{cs}(\beta)$ contains non-univalent functions, infact for $0\leq\beta \leq 1/2$ the class $\mathcal{S}_{cs}(\beta)\subset \mathcal{S}^{*}.$ In 2022, Kumar et al. \cite{Kumar n Yadav(2022)} introduced the class $\mathcal{F}(A,B)=\mathcal{F}(\psi_{4})$ and established some radii results.  
\begin{figure}[ht]
\begin{framed} 
   \centering
 \subfloat[\centering]{\includegraphics[width=0.30\textwidth]{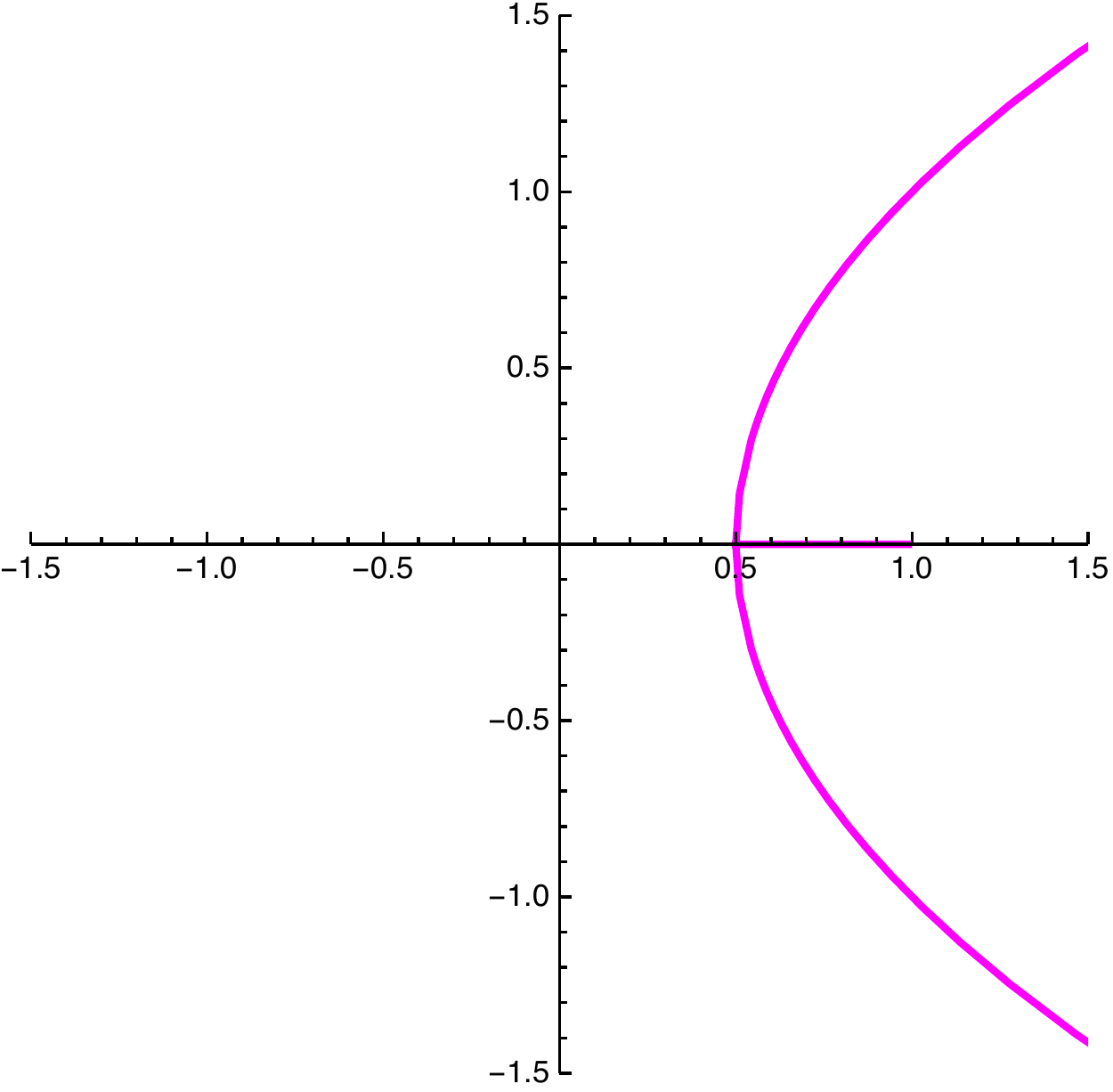}}\qquad \qquad \qquad \qquad
  \subfloat[\centering]{\includegraphics[width=0.30\textwidth]{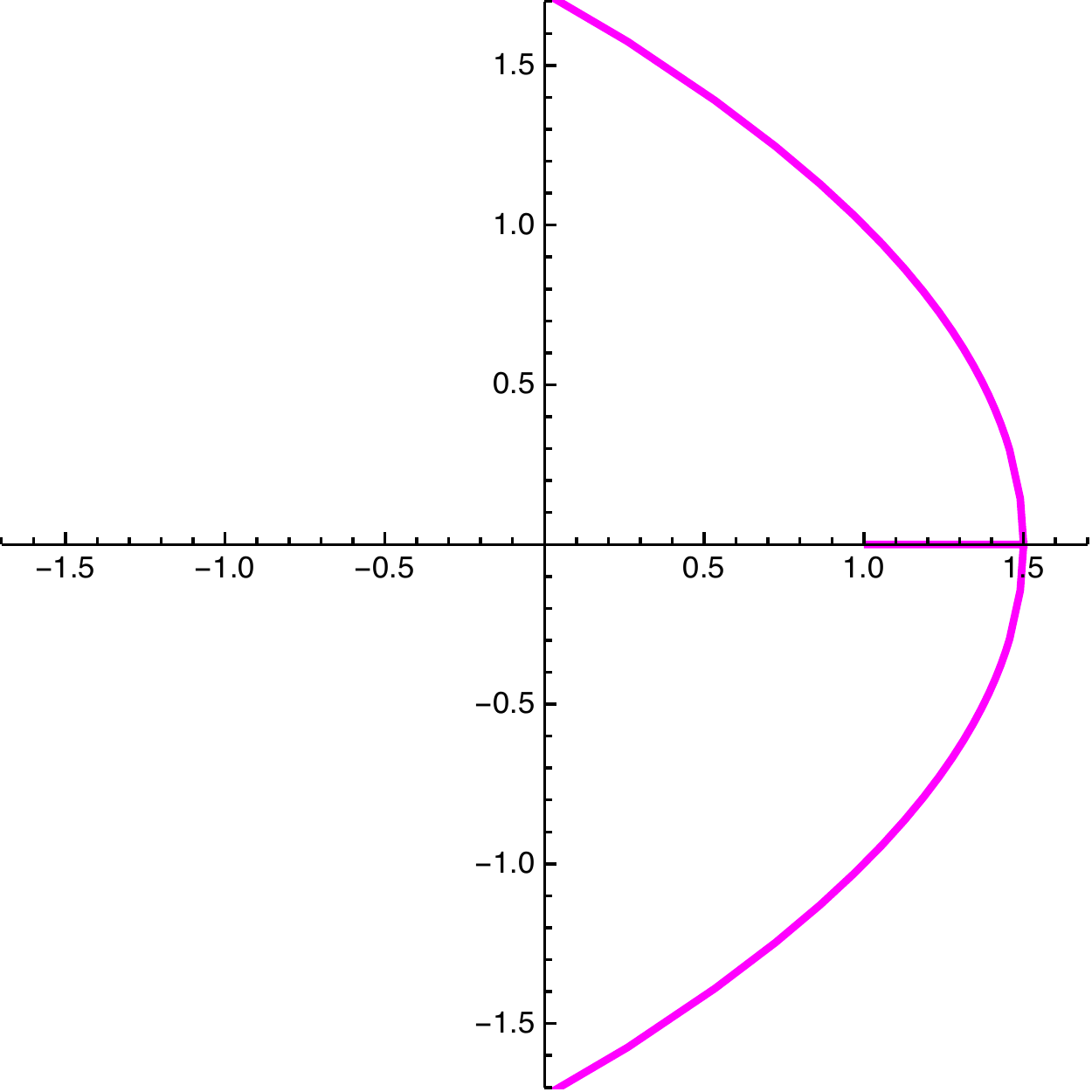}}\qquad \qquad \qquad  \qquad
  \caption{Graphs representing different parabolas with a common focus (1,0) \\ (a) $1+ \mathcal{P}_{0,\pi}(\partial\mathbb{D}),$ (b) $1+ \mathcal{P}_{0,0}(\partial\mathbb{D})$.}
\label{fig:7}
\end{framed}
\end{figure}  
\\ Motivated  essentially by the above classes and observations, we now study a subclass of $\mathcal{A}$ containing non-univalent functions. For $\tau,\theta\in(-\pi,\pi],$ the transformation \[\omega=((2e^{i\theta/2}\sqrt{2}/\pi)(\tan^{-1}(e^{i(\tau - \pi/2)}\sqrt{z})))^{2}\] maps the boundary of $\mathbb{D}$ onto a parabola (see \textbf{Fig.} \ref{fig:7}), given by 
\small{
\begin{align*}
{\mathcal{P}}_{\tau,\theta}(z)&:=\frac{2e^{i(\theta+\pi)}}{\pi^{2}}\left(\log\left(\frac{1+e^{i\tau} \sqrt{z}}{1-e^{i\tau}\sqrt{z}}\right)\right)^{2} \nonumber \\& =  \frac{8e^{i(\theta+\pi)}}{\pi^{2}} \sum _{n=1}^\infty \left(\frac{e^{2i\tau n}}{n}\sum _{k=0}^{n-1} \frac{1}{2 k+1}\right)z^{n},
\end{align*}}where the branch of $\sqrt{z}$ is chosen so that $\operatorname{Im}\sqrt{z}\geq 0.$ Note that the function $\mathcal{P}_{\tau,\theta}(z)$ is univalent in $\mathbb{D}$ and $1+\mathcal{P}_{0,\pi}(z)$ (see \textbf{Fig.} \ref{fig:7}(a)) is the parabolic function introduced independently by Ma-Minda \cite{Ma n Minda(1993)} and Ronning \cite{Ronning(UCV 1993)}, and used it 
extensively in context of parabolic starlike functions (see \cite{Kanas(2000),Ronning(UCV 1993)}). Further, Kanas \cite{Kanas(2003)} discussed some differential subordination techniques  and later  studied admissibility results \cite{Kanas(2006)} involving $1+\mathcal{P}_{0,\pi}(z).$ Note that the boundary of the function $\mathcal{P}_{0,0}(z)$ is a horizontal parabola with an opening in the left-half plane (see \textbf{Fig.} \ref{fig:7}(b)). In the present study, we consider the function $1+\mathcal{P}_{0,0}(z),$ in context of  non-univalent functions, which is completely different from the way Ronning, Ma-Minda and Kanas (see \cite{Kanas(2006),Kanas(2003),Ma n Minda(1993),Ronning(UCV 1993)}) handled parabolic regions. Obviously, the other choices of $\tau$ and $\theta$ leads to oblique parabolic regions, which is beyond the scope of our present study and therefore it is skipped here, as it needs to be handled differently.

For brevity, let us assume ${\mathcal{P}_{0}}(z):={\mathcal{P}_{0,0}}(z)$ and $\mathcal{LP}(z):=1+{\mathcal{P}_{0}}(z),$ then the horizontal parabolic region $\mathcal{LP}(z)$ (see \textbf{Fig.} \ref{fig:1}), is given by
\begin{align*}
   \Omega_{\mathcal{LP}}:=\{\omega\in\mathbb{C}:(\operatorname{Im}\omega)^{2}<3 - 2 \operatorname{Re}\omega \text{ or } |1-\omega|<2-\operatorname{Re}\omega\}.
\end{align*}
Since the major part of the above region lies in the left half plane, it is interesting to find the optimal radius of the domain disc for which it is fully mapped into the right half plane. Here below we define a class of analytic functions consisting of non-univalent functions. 
\begin{figure}
\begin{framed}
\includegraphics[height=2.2in,width=4.4in]{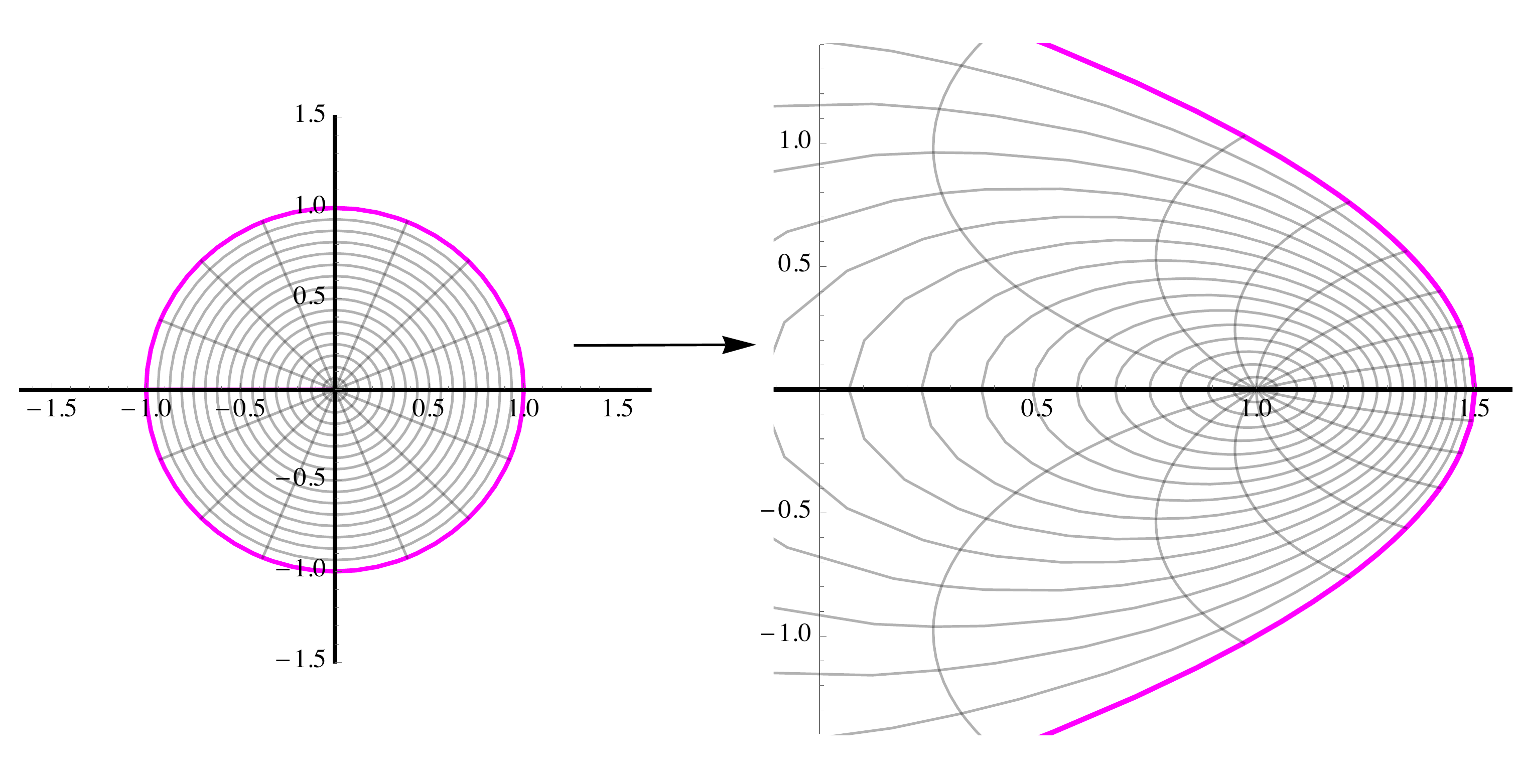}
\caption {Image of open unit disk $\mathbb{D}$ under the mapping $\mathcal{LP}(z).$}
\label{fig:1}
\end{framed}
\end{figure}
\begin{definition}
Let the class $\mathcal{F}_{\mathcal{LP}}$ consist of functions $f\in\mathcal{A}$ satisfying the subordination 
\[\frac{zf'(z)}{f(z)}\prec  \mathcal{LP}(z)=1-\frac{2}{\pi^{2}}\left(\log\left(\frac{1+\sqrt{z}}{1-\sqrt{z}}\right)\right)^{2}.\]
\end{definition}
Note that if $f_{0}\in\mathcal{F}_{\mathcal{LP}},$ then it can be expressed as 
\begin{equation}\label{EQN42}
    f_{0}(z) = z  \left( \exp \int_{0}^{z} \frac{\mathcal{P}_{0}(t)}{t}dt \right),
\end{equation}\noindent which acts as an extremal function for many radius results. Evidently, the subclasses of starlike functions $\mathcal{S}^{*}_{SG},$ $\mathcal{S}^{*}_{\mathcal{RL}}$ and $\mathcal{SL}^{*}(\alpha)$ $(0\leq\alpha<1)$ (see Table \ref{EQNTable2}) 
are contained in $\mathcal{F}_{\mathcal{LP}}
,$ however $\mathcal{F}_{\mathcal{LP}}\nsubseteq \mathcal{S}.$  
In section 2, we examine some geometrical properties of the function $\mathcal{LP}(z).$

\section{Main Results}

\subsection{Geometric Properties of the function $\mathcal{LP}(z)$} 
In Lemma \ref{L41}, we establish the maximum and minimum bounds of real part of the function $\mathcal{P}_{0}(z).$
\begin{lemma}\label{L41}
Let $z\in\mathbb{D}_{r}=\{z:|z|=r\},$ then for each $0\leq r<1$ and $\alpha\in(-\pi,\pi],$ we have
\[{\mathcal{P}}_{0}(r)\leq\operatorname{Re}{\mathcal{P}}_{0}(r e^{i\alpha})\leq {\mathcal{P}}_{0}(-r).\] 
\end{lemma}
\vskip -1cm
\begin{proof}
Suppose $z=re^{i\alpha},$ where $-\pi<\alpha\leq \pi,$ then for $|z|=r<1,$
\begin{align*}
\operatorname{Re}(\mathcal{P}_{0}(z))&=-\frac{2}{\pi^{2}}\left\{\operatorname{Re}\left(\log\left(\frac{1+\sqrt{z}}{1-\sqrt{z}}\right)\right)^{2}\right\}\\&=-\frac{2}{\pi^{2}}\left(\log\left(\sqrt{\frac{\mu_{1}(r,c)}{\mu_{2}(r,c)}}\right)\right)^{2}+\frac{2}{\pi^{2}}\left(\tan ^{-1}\left(\frac{2\sqrt{1-c^{2}} \sqrt{r}}{1-r}\right)\right)^2\\&=:\mathcal{G}(r,c).
\end{align*} 
where $c:=\cos (\alpha/2)$ and 
\begin{align*}
\mu_{i}(r,c):=
\left\{    
    \begin{array}{ll}
     1+r+2 c\sqrt{r}, & i=1, \\
     1+r-2 c\sqrt{r},  &  i=2.
    \end{array}
\right.
\end{align*}
Observe that $c\in[-1,1],$ infact it is easy to check that $\partial \mathcal{G}(r,c)/\partial c = 0$ if and only if $c=0,$ also $\partial^{2} \mathcal{G}(r,0)/\partial c^{2} <0,$ which leads to
\begin{align}\label{EQN418}
    \displaystyle{\max _{c\in[-1,1]}\mathcal{G}(r,c)=\mathcal{G}(r,0)}=\mathcal{P}_{0}(-r)=1+\frac{2}{\pi^{2}}\left(\tan^{-1}\left(\frac{2\sqrt{r}}{1-r}\right)\right)^{2}.
\end{align} Moreover, for each $R\leq r<1,$ equation \eqref{EQN418} leads to $\mathcal{G}(r,0)\geq \mathcal{P}_{0}(r)=\mathcal{G}(r,1).$
Since $\mathcal{G}(r,0)$ is an increasing function, whereas $\mathcal{G}(r,1)$ is a decreasing function of $r,$ this leads to the inequality $\mathcal{G}(r,1) < \mathcal{G}(r,0),$ for each $r<1$. Hence the required bound is achieved. 
\end{proof}

As a consequence of Lemma \ref{L41} and \cite[Theorem 2.1 \& Corollary 2.2]{Gangania n Kumar(2021)Trans}, 
we obtain the Growth and Covering Theorems for the class $\mathcal{F}_{\mathcal{LP}}.$
\begin{theorem}\label{EQN416}
Let $f\in\mathcal{F}_{\mathcal{LP}},$ then the following holds
\begin{enumerate}[I.]
 \item \it{(Growth Theorem) For $|z|=r<1,$ let  \small{\[\max_{|z|=r}\operatorname{Re}\mathcal{P}_{0}(z)=\mathcal{P}_{0}(-r)\text{ and }  \min_{|z|=r}\operatorname{Re}{\mathcal{P}_{0}}(z)=\mathcal{P}_{0}(r),\]} then for $|z|=r<1$ the following sharp inequality holds
\begin{align*} 
     r \exp\left(\int_{0}^{r}\frac{\mathcal{P}_{0}(t)}{t} dt \right)\leq |f(z)|\leq r \exp\left(\int_{0}^{r}\frac{\mathcal{P}_{0}(-t)}{t} dt \right).
\end{align*}
\item (Covering Theorem) 
Suppose $\min_{|z|=r}\operatorname{Re}{\mathcal{P}_{0}}(z)=\mathcal{P}_{0}(r)$ and $f\in\mathcal{F}_{\mathcal{LP}}.$ Let $f_{0}$ be given by \eqref{EQN42}, then $f(z)$ is a rotation of $f_{0}$ or $\left\{w\in\mathbb{D}:|w|\leq -f_{0}(-1)\right\}$ 
$\subset f(\mathbb{D}),$
where $-f_{0}(-1)=\lim_{r\to 1} -f_{0}(-1).$}
\end{enumerate}
\end{theorem}

\begin{remark}
(See \textbf{Fig.} \ref{fig:3})  If $f(z)$ is of the form $f(z)=z+a_{2}z^{2}+a_{3}z^{3}+\ldots,$ belongs to $\mathcal{F}_{\mathcal{LP}},$ then 
\[\left|\arg\left(\frac{zf'(z)}{f(z)}-2\right)\right|>\frac{3\pi}{4}.\]
It can be verified that the equation of tangent corresponding to $\Gamma:y^{2}=1+2(1-x)$ is given by $y=\pm (x-2).$ Infact these tangents intersect the parabola $\Gamma$ at the points $(1,\pm 1).$ Therefore it can be observed that the convex region $\Omega_{\mathcal{LP}}$ lies in the sector $|\arg(\omega-2)|>3\pi/4.$ Hence this gives a sharp argument estimate for functions lying in the class $\mathcal{F}_{\mathcal{LP}}.$ 
\end{remark}

\begin{remark}\label{EQN404}
Due to Lemma \ref{L41}, for $|z|=r<1,$  we have $\mathcal{LP}(r)\leq \operatorname{Re}\mathcal{LP}(z)\leq \mathcal{LP}(-r)$ and infact $\max_{|z|=r}|\mathcal{LP}(z)|=|\mathcal{LP}(r)|=|\mathcal{P}_{0}(r)|.$ \end{remark} 

\subsection{Radius Problems for the class $\mathcal{F}_{\mathcal{LP}}$}
Based on the definition of the class $\mathcal{F}_{\mathcal{LP}}$ and pictorial representation of $\mathcal{LP}(\partial \mathbb{D})$ (see \textbf{Fig.} \ref{fig:1}), we have $\max_{|z|\leq 1} \operatorname{Re}\left(\mathcal{LP}(z)\right)=\mathcal{LP}(-1)=3/2.$ This means $\operatorname{Re}zf'(z)/f(z)<3/2,$ thus $f\in \mathcal{F}_{\mathcal{LP}}$ may or may not be a univalent function. Therefore it is an interesting problem to establish the largest radius $r_{0}<1$ such that each $f\in\mathcal{F}_{\mathcal{LP}}$ is starlike in $|z|\leq r_{0}.$ 
In this section, we study some radius results for the class $\mathcal{F}_{\mathcal{LP}}$ along with the classes $\mathcal{S}^{*}(\phi)$ and $\mathcal{F}(\psi)$ for some special choices of $\phi(z)$ and $\psi(z),$ as mentioned in Table \ref{EQNTable2} (see Appendix)  
and equation \eqref{EQN420}, respectively. Here below we provide a lemma that yields a maximal disc that can be subscribed within the parabolic region $\Omega_{\mathcal{LP}}.$

\begin{lemma}\label{EQN406}
Suppose $a< 3/2$ and  assume that $\zeta(\eta)$ is defined as follows: \[\zeta=\zeta(\eta)=\log \left(\dfrac{\sqrt{\eta}}{\sqrt{1-\eta}}\right) \text{ with } \eta=\frac{e^{-\pi\sqrt{1-2 a}}}{1+e^{-\pi\sqrt{1-2 a}}},\] then $\mathcal{LP}(\mathbb{D})$ satisfies the following inclusion 
\[\mathcal{D}(a,r_{a}):=\left\{\omega\in \mathbb{C}:|\omega-a|<r_{a}\right\}\subset \Omega_{\mathcal{LP}},\]
where 
\begin{align*}
r_{a}=\left\{\begin{array}{cl}
  \sqrt{\left(a-\dfrac{3}{2}+\dfrac{2 \zeta^2}{\pi ^2}\right)^{2}+\dfrac{4 \zeta^2}{\pi ^2}}, & a \leq \dfrac{1}{2}\\
  \dfrac{3}{2}-a, & \dfrac{1}{2}<a<\dfrac{3}{2},
\end{array}\right.
\end{align*}
\end{lemma}
\begin{proof}
We obtain a maximal disc centered at $(a,0),$ where $a< 3/2,$ that can be inscribed within $\Omega_{\mathcal{LP}}.$ The distance from center $(a,0)$ to the boundary $f(\partial (\mathbb{D}))$ is given by square root of 
\[\mathcal{D}_{a}(X):=\left(a+\frac{2}{\pi^{2}} \left(\log \left(\frac{\sqrt{X^2}}{\sqrt{1-X^2}}\right)\right)^{2}-\frac{3}{2}\right)^2+\frac{4}{\pi^{2}} \left(\log \left(\frac{\sqrt{X^2}}{\sqrt{1-X^2}}\right)\right)^{2},\] where $X=\cos t .$ Now the critical points of $\mathcal{D}_{a}(X)$ are  
\begin{align*}
X':=
  \begin{cases} 
    \pm \dfrac{e^{\frac{1}{2} \pi  \sqrt{1-2 a}}}{\sqrt{1+e^{\pi  \sqrt{1-2 a}}}}, \pm\dfrac{e^{-\frac{1}{2} \pi  \sqrt{1-2 a}}}{\sqrt{1+e^{-\pi \sqrt{1-2 a}}}}, &\text{if } a<1/2,  \\ & \\
      \pm 1/\sqrt{2}, &\text{if }   1/2\leq a<3/2.
 \end{cases}
\end{align*} 
It can be verified that $\mathcal{D}_{a}''(X)>0$ at $X=X',$ whenever  $a<3/2.$ Therefore, $X=X'$ is the point of minima for $\mathcal{D}_{a}(X),$ which leads us to the optimal disk centered at $a$ with radius $r_{a}.$
\end{proof}

\begin{theorem}\label{EQN46}
Suppose $0\leq \alpha <1$ and $-1<B<A\leq1,$ then for $f\in\mathcal{A},$ the sharp $\mathcal{F}_{\mathcal{LP}}-$radii for the classes $\mathcal{S}^{*}_{p},$ $\mathcal{S}^{*}_{s},$ $\Delta^{*},$ $\mathcal{S}^{*}_{\varrho},$  $\mathcal{S}^{*}_{\rho},$ $\mathcal{S}^{*}_{\wp},$ $\mathcal{BS}^{*}(\alpha),$ $\mathcal{S}^{*}_{\alpha,e}$  and $\mathcal{S}^{*}(A,B)$ (see Table \ref{EQNTable2} in Appendix) are respectively given by  
\begin{enumerate}[(i)]
\item $\mathcal{R}_{\mathcal{F}_{\mathcal{LP}}}(\mathcal{S}^{*}_{p})=\tanh ^2(\pi /4).$
\item $\mathcal{R}_{\mathcal{F}_{\mathcal{LP}}}(\mathcal{S}^{*}_{s})=\pi/6.$ 
    \item $\mathcal{R}_{\mathcal{F}_{\mathcal{LP}}}(\Delta^{*})=5/12.$
\item $\mathcal{R}_{\mathcal{F}_{\mathcal{LP}}}(\mathcal{S}^{*}_{\varrho})=(\cosh ^{-1}(3/2))^2.$
\item $\mathcal{R}_{\mathcal{F}_{\mathcal{LP}}}(\mathcal{S}^{*}_{\rho})=\sinh (1/2).$
    \item $\mathcal{R}_{\mathcal{F}_{\mathcal{LP}}}(\mathcal{S}^{*}_{\wp})\approx 0.3517\ldots.$ 
    \item For $0<\alpha<1,$ $\mathcal{R}_{\mathcal{F}_{\mathcal{LP}}}(\mathcal{BS}^{*}(\alpha))=R_{\mathcal{BS}},$ where 
\begin{align*}
R_{\mathcal{BS}}=\left\{\begin{array}{cl}
  1/2, & \alpha = 0\\
(\sqrt{1+\alpha}-1)/\alpha, & 0 < \alpha < 1 .
\end{array}\right.
\end{align*}
    \item $\mathcal{R}_{\mathcal{F}_{\mathcal{LP}}}(\mathcal{S}^{*}_{\alpha,e})={R}_{\alpha,e},$ where 
 \begin{align*}
R_{\alpha,e}=\left\{\begin{array}{cl}
  \log(1-1/2(\alpha-1)), & 0\leq \alpha < 1-1/2(e-1)\\
 1, & 1-1/2(e-1)\leq \alpha < 1.
\end{array}\right.
\end{align*}
In particular,
    $\mathcal{R}_{\mathcal{F}_{\mathcal{LP}}}(\mathcal{S}^{*}_{e})=\log(3/2).$
    \item $\mathcal{R}_{\mathcal{F}_{\mathcal{LP}}}(\mathcal{S}^{*}(A,B))=:\tilde{R},$ where 
    \begin{align*}
    \tilde{R}=
  \begin{cases} 
    1/(2A-3B),& \text{when } ((- 1 < B \leq (2 A - 1)/3) \wedge (-1<A<0) ) \\& \qquad \lor ((- 1 < B < (2 A - 1)/3) \wedge (0\leq A \leq1)),  \\  
   1, &  \text{when } ((2 A - 1)/3<B<A\leq 1) \wedge (-1<A<0 ))\\& \quad \lor( ( (2 A - 1)/3\leq B < A \leq 1) \wedge (0\leq A \leq1)).
\end{cases}
 \end{align*}
\end{enumerate} 
\end{theorem}
\begin{proof}
For part $(i),$ as $f\in\mathcal{S}^{*}_{p},$ then due to the geometry of the function $1+\mathcal{P}_{0,\pi}(z)=1+2/\pi^{2}(\log((1+\sqrt{z})/(1-\sqrt{z})))^{2}$ it can be observed that   \[\displaystyle{\max_{|z|=r}\operatorname{Re}(1+\mathcal{P}_{0,\pi}(z))} = 1+\mathcal{P}_{0,\pi}(r).\]
\begin{figure}[ht]
\begin{framed}
 \begin{tabular}{c}
\includegraphics[height=2.5in,width=1.9in,angle=0]{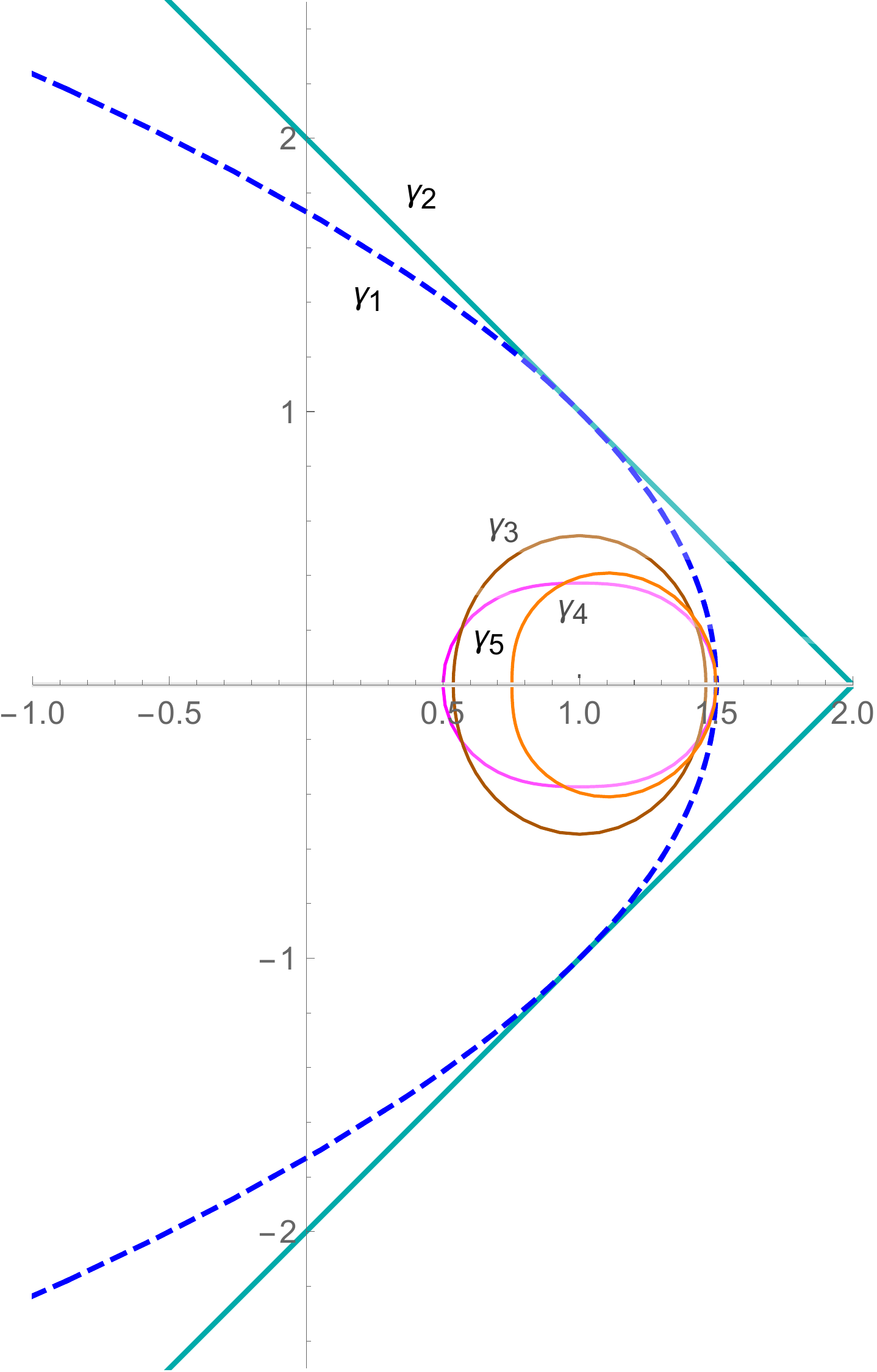}
\end{tabular}
\begin{tabular}{l}
\parbox{0.25\linewidth}{
$\gamma_{1}: \mathcal{LP}(z)$}\\ \vspace{0.15cm}  \hspace{0.45cm} $1-2/\pi^{2}(\log((1+\sqrt{z})/(1-\sqrt{z}))^{2}$ \\ \vspace{0.15cm}  
$\gamma_{2}:  |\arg (u-2)| = 3\pi/4$\\ \vspace{0.15cm} 
$\gamma_{3}:$ Modified Sigmoid function \\  \vspace{0.15cm}  \hspace{0.45cm} $2/(1+e^{-z})$  \\ \vspace{0.15cm}
$\gamma_{4}:$ Cardioid function \\ \vspace{0.15cm} \hspace{0.47cm}
$1+z e^z $ in $|z|<r^{*},$ \\ \vspace{0.15cm} \hspace{0.47cm}$r^{*}\approx 0.351734$ \\ \vspace{0.15cm} 
$\gamma_{5}: $ Booth Lemniscate for $\alpha=0.8$ \\ \vspace{0.15cm} \hspace{0.47cm} $1+5z/(5-4 z^2)$ in  $|z|<r_{0},$  \\ \vspace{0.15cm} \hspace{0.47cm} $r_{0}=|z|<(1/4) (3 \sqrt{5}-5).$
\end{tabular}
\caption{Pictorial boundary description of inclusion results pertaining to the parabolic function $\mathcal{LP}(z)$} 
\label{fig:3}
\end{framed}
\end{figure}
Now for $f(z)$ to lie in the class $\mathcal{F}_{\mathcal{LP}},$ we must have
\[1+\frac{2}{\pi^{2}}\left(\log\left(\frac{1+\sqrt{r}}{1-\sqrt{r}}\right)\right)^{2}\leq \frac{3}{2},\] which holds provided $r\leq\mathcal{R}_{\mathcal{F}_{\mathcal{LP}}}(\mathcal{S}^{*}_{p}).$ Clearly equality in $(i)$ is attained for the function $\tilde{f}(z)$ such that   
$z\tilde{f}'(z_{0})/\tilde{f}(z_{0})=1+\mathcal{P}_{0,\pi}(z_{0})$ at $z_{0}=\mathcal{R}_{\mathcal{F}_{\mathcal{LP}}}(\mathcal{S}^{*}_{p}).$ Further observe that, if $p\in\mathcal{P}[A,B]=\{p(z):p(z)=1+c_{1}z+c_{2}z^{2}\ldots \prec (1+Az)/(1+Bz),-1\leq B<A\leq 1\},$ then for $|z|=r<1,$ it is a known fact that   
\begin{equation}\label{EQN422}
\left|p(z)-\frac{1-AB r^{2}}{1-B^{2}r^{2}}\right|<\frac{|A-B|r}{1-B^{2}r^{2}}.
\end{equation} Then, in view of \eqref{EQN422}, for part $(ix)$ $f\in\mathcal{S}^{*}(A,B)$ lies in $\mathcal{F}_{\mathcal{LP}},$ if 
\[\frac{(A-B)r+1-A B r^{2}}{1-B^{2}r^{2}} \leq \frac{3}{2}.\]
Equivalently, we can say that 
$(1-B r) ((2 A -3 B) r-1)\leq 0,$ provided 
$r \leq \mathcal{R}_{\mathcal{F}_{\mathcal{LP}}}(\mathcal{S}^{*}(A,B)).$ Extremal function in this case is $\hat{f}\in\mathcal{A}$ satisfying
$z_{0}\hat{f}'(z_{0})/\hat{f}(z_{0})=(1+A z_{0})/(1+B z_{0}),$ where $z_{0}=\tilde{R}.$
For part $(iv),$ as $f\in\mathcal{F}_{\mathcal{LP}},$ then proceeding as before, $f$ lies in $\mathcal{S}^{*}_{\varrho}$ if for $|z|=r$ we have $3/2\geq\cosh{\sqrt{r}},$ provided $r=\mathcal{R}_{\mathcal{F}_{\mathcal{LP}}}(\mathcal{S}^{*}_{\varrho}).$ Sharpness holds for the function ${f}_{\varrho}\in\mathcal{A}$ defined as $z{f}'_{\varrho}(z)/{f}_{\varrho}(z)=\cosh \sqrt{z}.$ Further we know that $\max_{|z|=r} \RE (1+\sin z)=1+\sin r,$ then for part $(ii)$ it is enough to find an $r<1$ satisfying the equation $ \sin r=1/2,$ thus $r=\mathcal{R}_{\mathcal{F}_{\mathcal{LP}}}(\mathcal{S}^{*}_{s})=\pi/6.$  Sharpness holds for the function ${f}_{s}(z)$ given by $z{f}'_{s}(z)/{f}_{s}(z)=1+\sin z.$
In all the subsequent parts, the proofs follow along the same lines, therefore they are omitted.
\end{proof}

Let $\mathcal{P}_{\alpha}$ consist of functions of the form 
$p(z)=1+c_{1}z+c_{2}z^{2}+\ldots,$ satisfying $\RE p(z)>\alpha$ for $0\leq\alpha<1,$ then we say $p(z)$ is a  Carath\'{e}odory function of order $\alpha.$ Denote $\mathcal{P}(0)=:\mathcal{P},$ commonly known as the class of Carath\'{e}odory functions. Further, assume $\mathfrak{P}_{\mathcal{LP}}$ to be the class of functions of the form $p(z)=1+c_{1}z+c_{2}z^{2}+\ldots,$ such that $p(z)\prec \mathcal{LP}(z).$

\begin{theorem}\label{EQN43}
Let $0\leq \alpha <1$ and $0\leq \gamma \leq \gamma_{\alpha},$ where $\gamma_{\alpha}=\tanh ^2(\pi  \sqrt{1-\alpha }/2 \sqrt{2}).$ If $p\in\mathfrak{P}_{\mathcal{LP}},$ then $p\in\mathcal{P}_{\alpha},$ i.e $p(z)$ is a Carath\'{e}odory function of order $\alpha,$  in the disc $|z|<\gamma_{\alpha}.$ 
 \end{theorem}
\vskip -2cm
 \begin{proof}
Since $p\in\mathfrak{P}_{\mathcal{LP}},$ then by definition of subordination and Schwarz Lemma, there exists an analytic function $w(z)$ with $|w(z)|\leq|z|<1$ and $w(0)=0,$ such that
$p(z)=\mathcal{LP}(w(z)).$ Suppose $w(z)=Re^{i\theta}$ $(-\pi<\theta\leq \pi),$ then $|w(z)|=R\leq |z|=r<1.$  On applying Lemma \ref{L41} for $p\in\mathfrak{P}_{\mathcal{LP}},$ we get
$\operatorname{Re} p(z)\geq {\mathcal{LP}}(r).$ Further  $p\in\mathfrak{P}_{\mathcal{LP}}$ is Carath\'{e}odory of order $\alpha$ $(0\leq \alpha<1)$ if ${\mathcal{LP}}(r)\geq \alpha,$ provided $r\leq \gamma_{\alpha}=\tanh^{2}(\pi\sqrt{1-\alpha}/2\sqrt{2}).$ The function $f_{0}(z)$ given by \eqref{EQN42}, is the extremal.
\end{proof}
\noindent Upon replacing $p(z)$ with $zf'(z)/f(z)$ in Theorem \ref{EQN43}, we deduce the next result.
 \begin{corollary}\label{EQN45}
 Let $0\leq \alpha <1$ and $0\leq \gamma \leq \gamma_{\alpha},$ where $\gamma_{\alpha}$ is as defined in Theorem \ref{EQN43}. If $f\in\mathcal{F}_{\mathcal{LP}},$ then $f(z)$ is starlike of order $\alpha$ in the disc $|z|<\gamma_{\alpha}.$ This result is sharp.
\end{corollary}
\begin{remark}\label{EQN49}
Put $\alpha=0$ in Theorem \ref{EQN43},  we get a sharp $\mathcal{P}-$ radius for the class $\mathfrak{P}_{\mathcal{LP}}.$ Infact for the class $\mathcal{F}_{\mathcal{LP}},$ Corollary \ref{EQN45} gives sharp radius of starlikeness $\gamma_{0}=\tanh ^2(\pi/2 \sqrt{2}).$ Moreover, $r=\gamma_{0}<1$ serves as the sharp radius of univalence for the class $\mathcal{F}_{\mathcal{LP}}.$
\end{remark}

\begin{theorem}\label{EQN409}
Assume  $0<\alpha\leq 1,$ then the sharp $\mathcal{S}^{*}(1+\alpha z)-$ radius for the class $\mathcal{F}_{\mathcal{LP}}$ is the unique positive root $r_{\alpha}=\tanh ^2(\pi  \sqrt{\alpha }/2\sqrt{2})$ of the equation
\begin{align}
    2\left(\log ((1+\sqrt{r})/(1-\sqrt{r}))\right)^{2}- \alpha \pi^{2}=0,
\end{align} 
where $\alpha$ is the radius of the disc $\left\{\omega:|\omega - 1|<\alpha\right\}.$
\end{theorem}
\vskip -2cm
\begin{proof}
In view of Remark \ref{EQN404}, for the circle $|z|=r<1,$ we have
\begin{align}\label{EQN47}
 \displaystyle{\max_{|z|\leq r<1}|\mathcal{LP}(z)|}=1-\frac{2}{\pi^{2}}\left(\log\left(\frac{1+\sqrt{r}}{1-\sqrt{r}}\right)\right)^{2} = \mathcal{LP}(r),
\end{align} 
which is a decreasing function. Infact $\mathcal{LP}(r)=0$ if and only if  $r=\tanh ^2(\pi/2\sqrt{2})\approx 0.6469\ldots.$ 
As $f\in {\mathcal{F}}_{\mathcal{LP}},$ then there exists a Schwarz's function 
$w(z)$ with $w(0)=0,$ so that 
\[\frac{zf'(z)}{f(z)}=\mathcal{LP}(w(z)).\] Assume $w(z)=Re^{i\theta}$ where $R\leq r<1.$ Now observe that for $0<\alpha\leq 1,$ equation \eqref{EQN47} yields
\[\left|\mathcal{LP}(z)-1\right|\leq |\mathcal{LP}(R)-1|\leq |\mathcal{LP}(r)-1|=|\mathcal{P}_{0}(r)|\leq \alpha, \] provided $r\leq \tanh ^2(\pi  \sqrt{\alpha }/2\sqrt{2})=r_{\alpha}.$ Further, at $z_{0}=r_{\alpha},$ the function ${f}_{0}(z)$  (defined in  \eqref{EQN42}) such that $z{f_{0}}'(z)/ f_{0}(z)=\mathcal{LP}(z),$ works as the extremal function.  
\end{proof}

As a consequence of Theorem \ref{EQN409}, $\mathcal{S}^{*}(1+\alpha z)-$ radii for some well-known Ma-Minda subclasses of starlike functions, namely, $\mathcal{S}^{*}_{e},$ $\mathcal{S}^{*}_{s},$ $\mathcal{S}^{*}_{\varrho},$ $\mathcal{S}^{*}_{\wp},$ $\mathcal{S}^{*}_{\rho},$ $\mathcal{S}^{*}_{SG}$ and $\mathcal{S}^{*}_{{N}_{e}}$ (see Table \ref{EQNTable2} in Appendix)
are stated in Corollary \ref{EQN410}. Moreover, sharpness of Corollary \ref{EQN410} is illustrated by \textbf{Fig.} \ref{fig:2}.

\begin{corollary}\label{EQN410}
Let $f\in\mathcal{A}$ belong to $\mathcal{F}_{\mathcal{LP}},$ then the following radii are sharp for the class $\mathcal{F}_{\mathcal{LP}},$ (see \textbf{Fig.} \ref{fig:2}) 
\begin{enumerate}[(i)]
    \item The $\mathcal{S}^{*}_{e}-$radius is $r_{1}=\tanh ^2( \lambda \pi),$ where $\lambda=(1/2) \sqrt{(e-1)/2 e}.$
    \item The $\mathcal{S}^{*}_{s}-$radius is $r_{2}=\tanh ^2(\pi/\lambda),$ where $\lambda=2 \sqrt{2 \csc 1}.$
    \item The $\mathcal{S}^{*}_{\varrho}-$radius is $r_{3}=\tanh ^2(\pi \lambda/ 2 ),$ where $\lambda=\sin(1/2).$
    \item The $\mathcal{S}^{*}_{\wp}-$radius is $r_{4}=\tanh^{2}(\pi/2\sqrt{2 e}).$
    \item The $\mathcal{S}^{*}_{\rho}-$radius is $r_{5}=\tanh ^2(\pi\sqrt{\lambda}/2),$ where $\lambda=(1/2) \sinh ^{-1}1.$ 
    \item The $\mathcal{S}^{*}_{SG}-$radius is $r_{6}=\tanh ^2(\lambda \pi/2\sqrt{2}),$ where $\lambda= \sqrt{(e-1)/(e+1)}.$
    \item The $\mathcal{S}^{*}_{{N}_{e}}-$radius is $r_{7}=\tanh ^2(\pi /2 \sqrt{3})$.
\end{enumerate} 
\end{corollary}

\vskip -0.79cm

\begin{figure}[H]
    \centering
  \subfloat[\centering]{\includegraphics[width=0.27\textwidth]{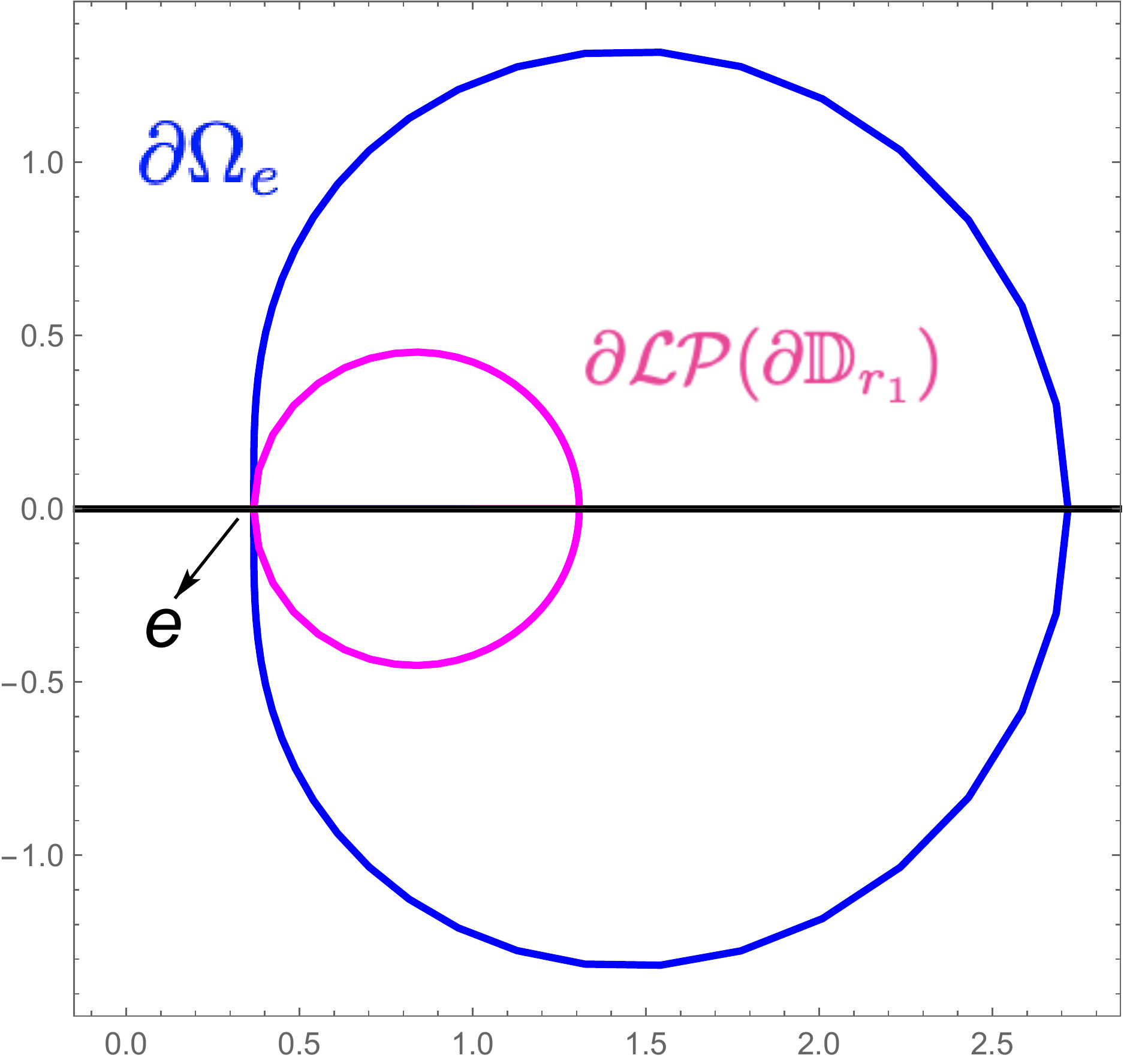}}  \qquad \quad
  \subfloat[\centering]{\includegraphics[width=0.21\textwidth]{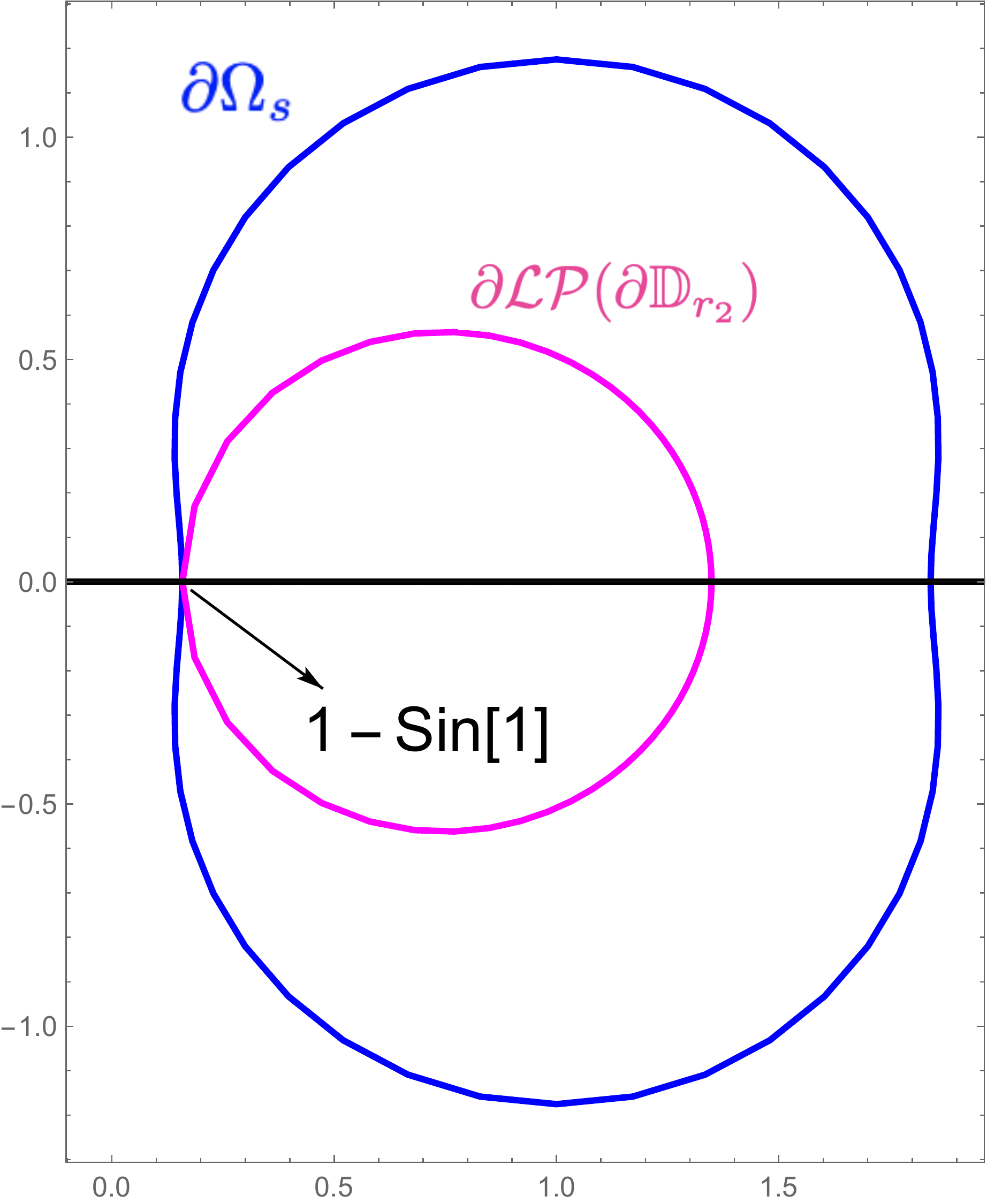}} \qquad
  \subfloat[\centering]{\includegraphics[width=0.36\textwidth]{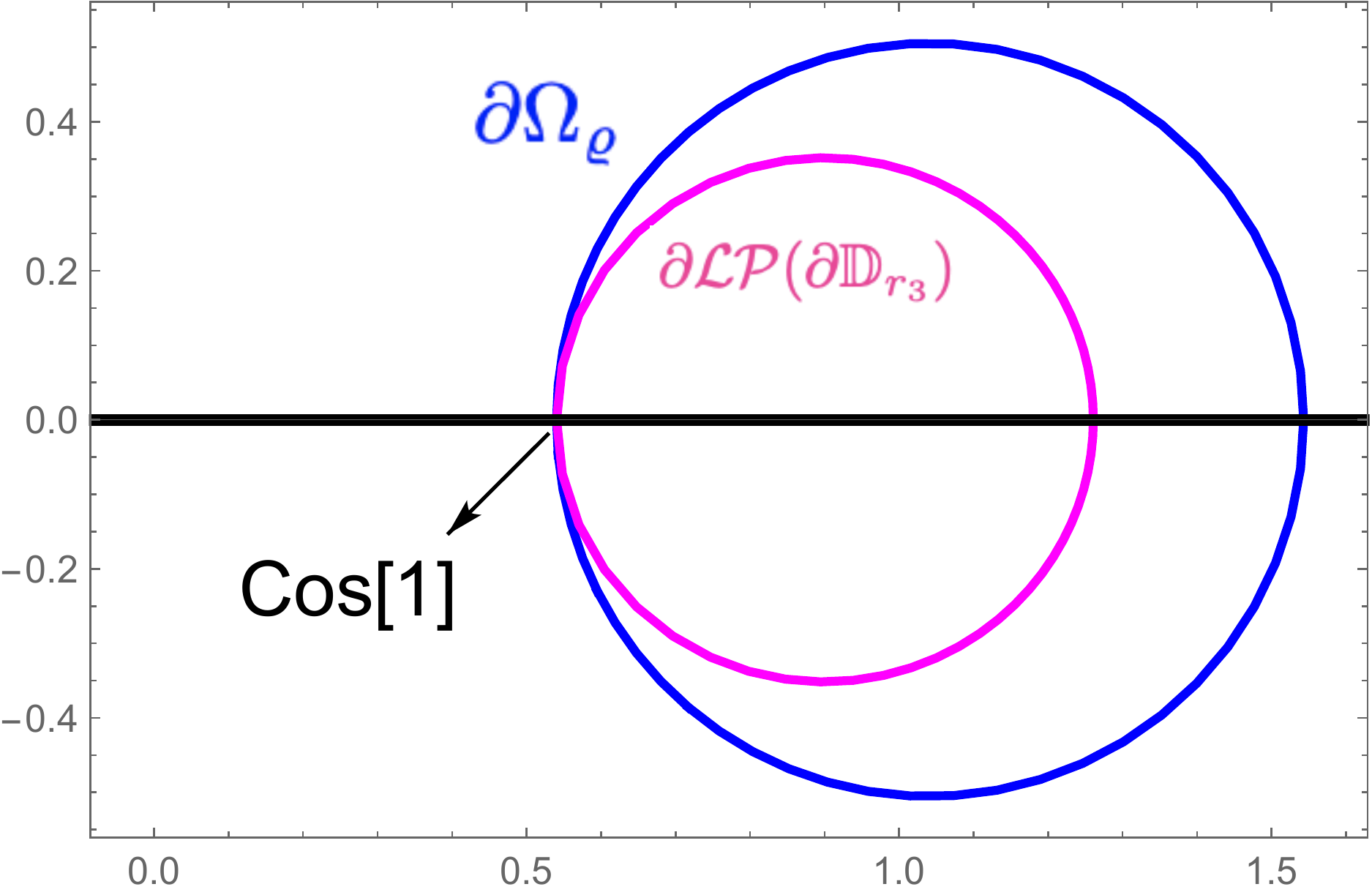}}
\end{figure}
\begin{figure}[H]
    \centering
  \subfloat[\centering]{\includegraphics[width=0.28\textwidth]{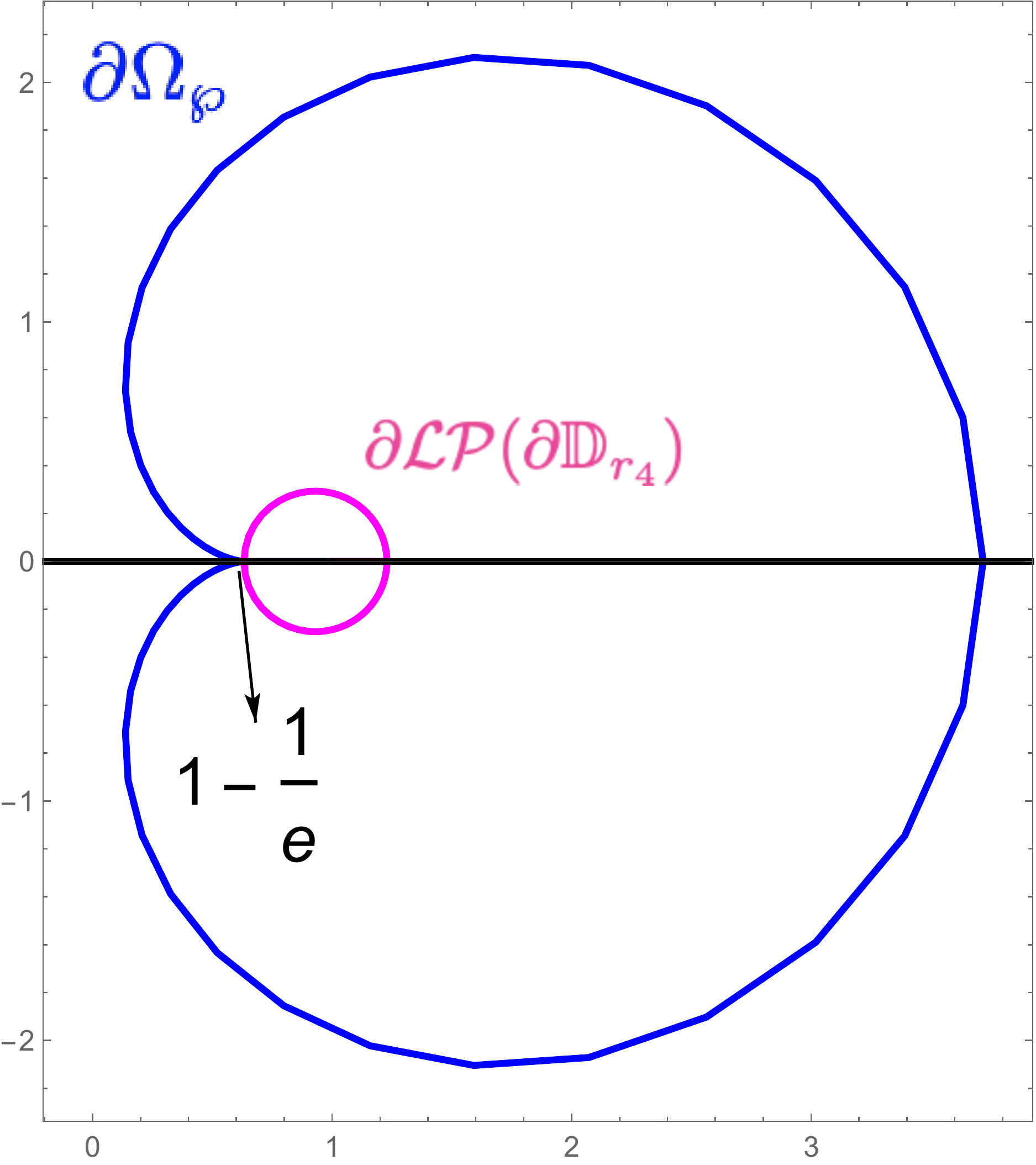}}\qquad \qquad
  \subfloat[\centering]{\includegraphics[width=0.21\textwidth]{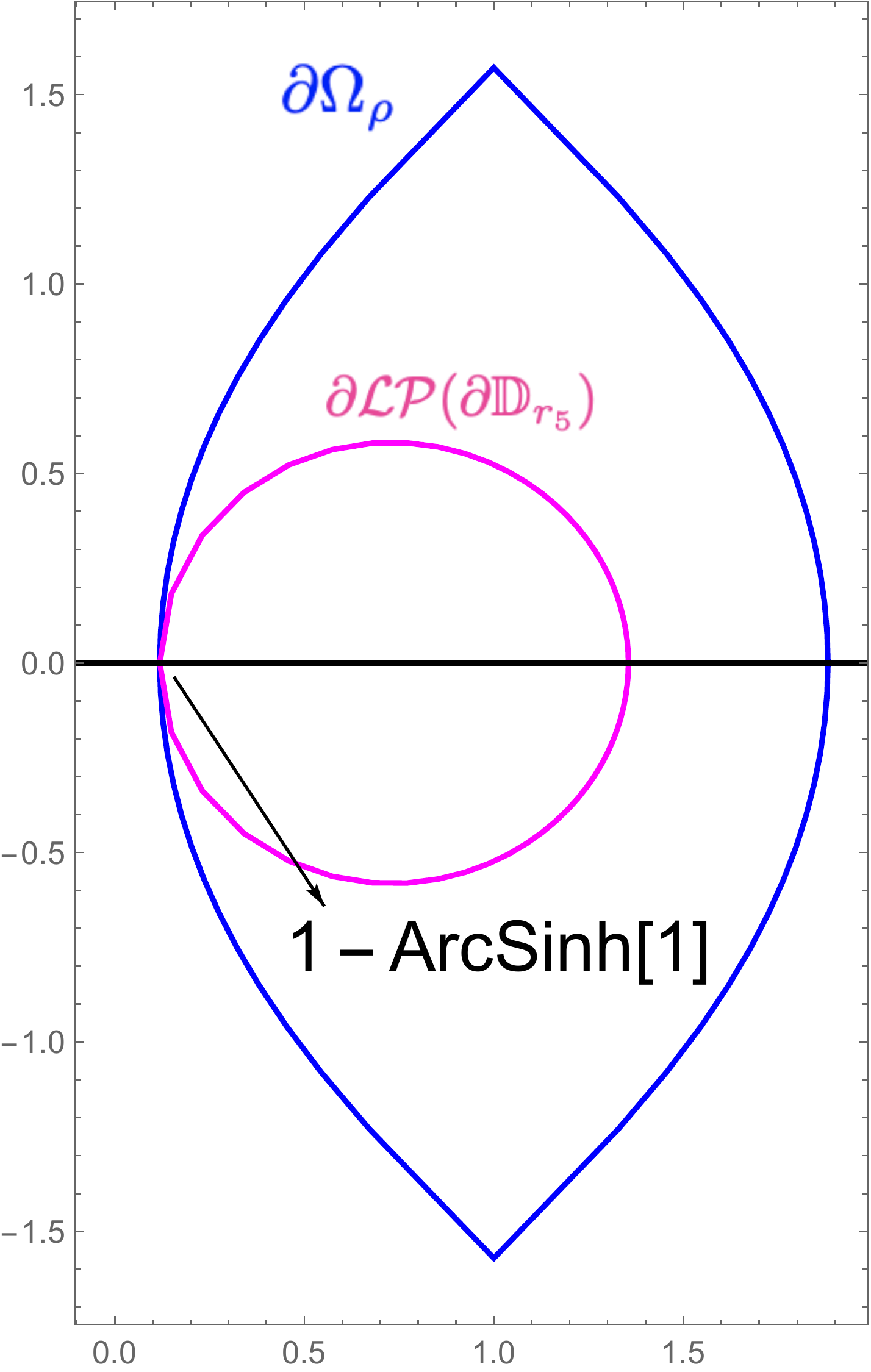}} \qquad
   \subfloat[\centering]{\includegraphics[width=0.34\textwidth]{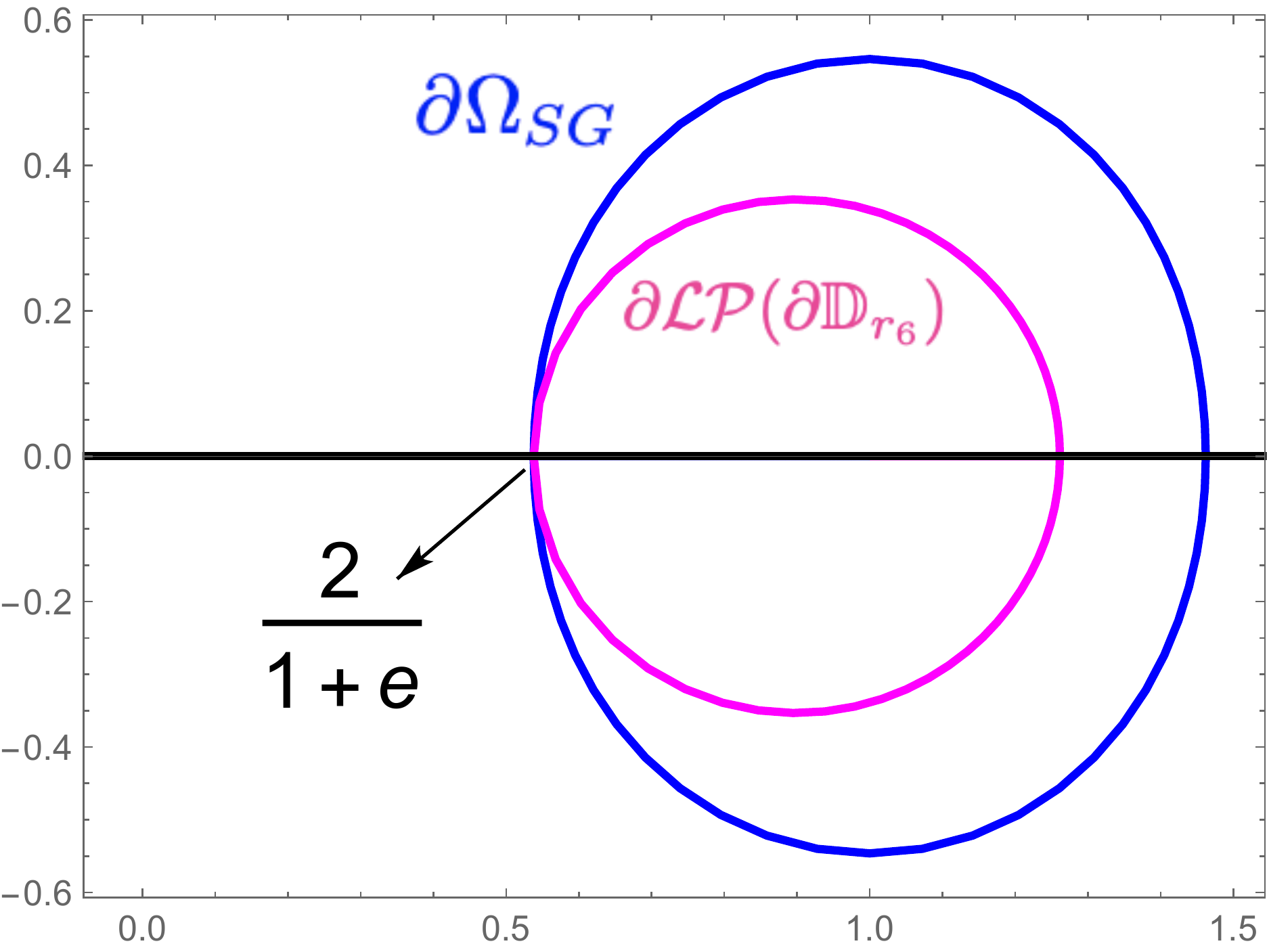}}\qquad  \quad \qquad
    \subfloat[\centering]{\includegraphics[width=0.21\textwidth]{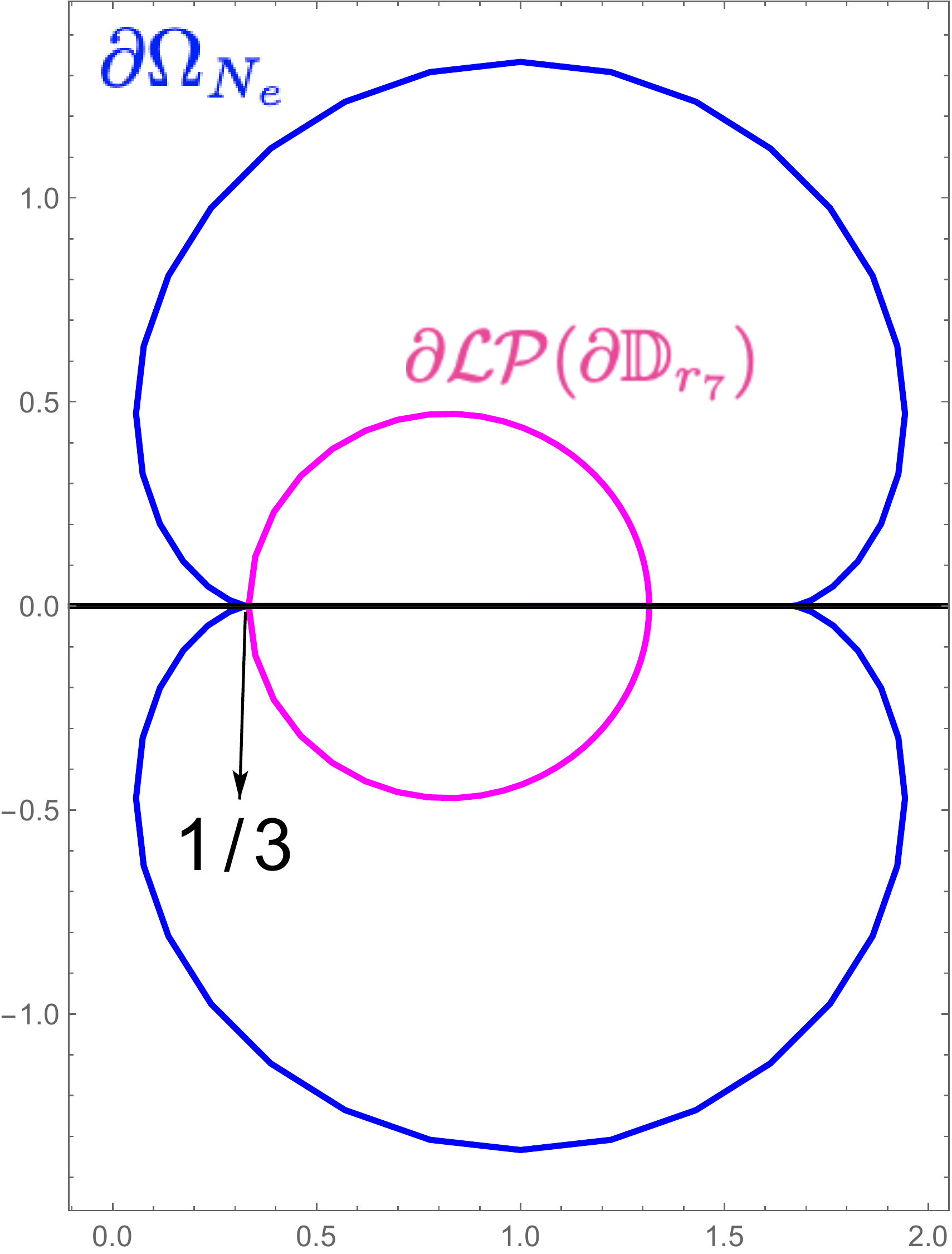}} \qquad  \qquad 
  \caption{Images depicting sharpness of radius result in Corollary \ref{EQN410}.}
\label{fig:2}
\end{figure}

In the next corollary we obtain the sharp $\mathcal{S}^{*}(\beta)-$ radius for the class $\mathcal{F}_{\mathcal{LP}}.$
\begin{corollary}\label{EQN408}
Suppose $0\leq\beta < 1$ and $f\in\mathcal{F}_{\mathcal{LP}},$ then sharp $\mathcal{S}^{*}(\beta)-$ radius is $\tanh ^2(\pi  \sqrt{\beta }/2 \sqrt{2}).$
\end{corollary}
\begin{remark}
On replacing $\beta$ with $1-\alpha$ $(0<\alpha\leq 1),$ in Corollary \ref{EQN408} we get the radius of starlikness of order $\alpha$ obtained in Corollary \ref{EQN45}. Moreover, from Corollary \ref{EQN408}, we obtain the sharp $\mathcal{S}^{*}_{\alpha}-$radius for the class $\mathcal{F}_{\mathcal{LP}},$ where $\mathcal{S}^{*}_{\alpha}=\left\{f\in\mathcal{A}:\left|zf'(z)/f(z)-1\right|<1-\alpha\right\}.$  
\end{remark}

\begin{corollary} \label{EQN415}
Let $\eta=\sqrt{2}-1$ and suppose $f\in\mathcal{F}_{\mathcal{LP}},$ then the following holds (see \textbf{Fig.} \ref{fig:6})
\begin{enumerate}[(i)]
    \item $f\in\mathcal{S}^{*}_{\mathcal{L}}$ in $|z|<\tanh ^2\left(\pi  \sqrt{\eta }/2 \sqrt{2}\right) \approx 0.376\ldots.$
    \item $f\in\mathcal{S}^{*}_{\mathcal{RL}}$ in $|z|<\tanh ^2(\pi  \sqrt[4]{\sqrt{2\ \eta } (1-\sqrt{2 \eta })}/{2 \sqrt{2}}) \approx 0.283\ldots.$
\end{enumerate}
\end{corollary}

\begin{remark}
From \textbf{Fig.} \ref{fig:6}, it is evident that the radii obtained in Corollary \ref{EQN415} can be further improved.
\end{remark} 
\begin{figure}[ht]
  \centering
\subfloat[\centering]{\includegraphics[width=0.35\textwidth]{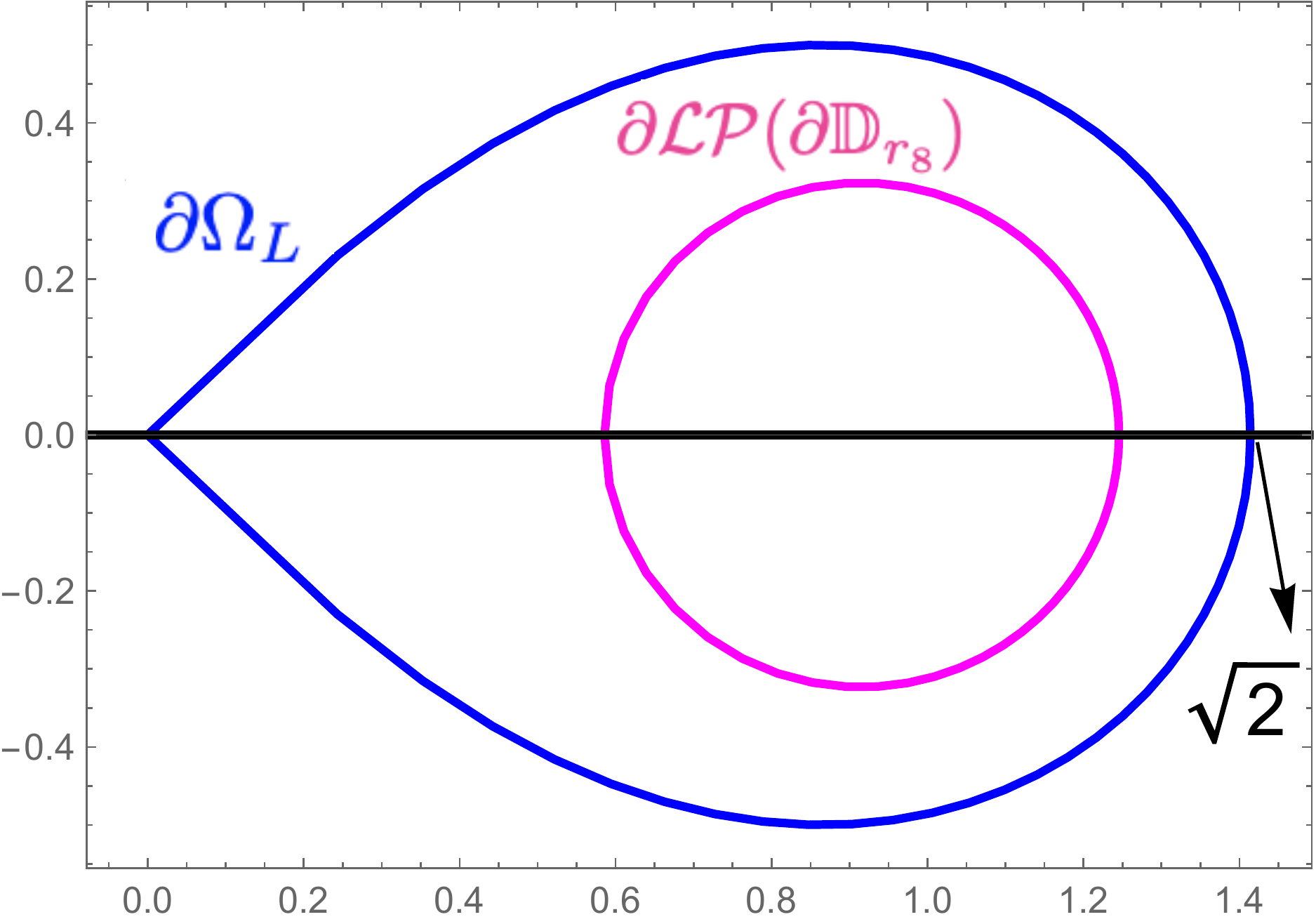}}  \qquad \qquad
  \subfloat[\centering]{\includegraphics[width=0.35\textwidth]{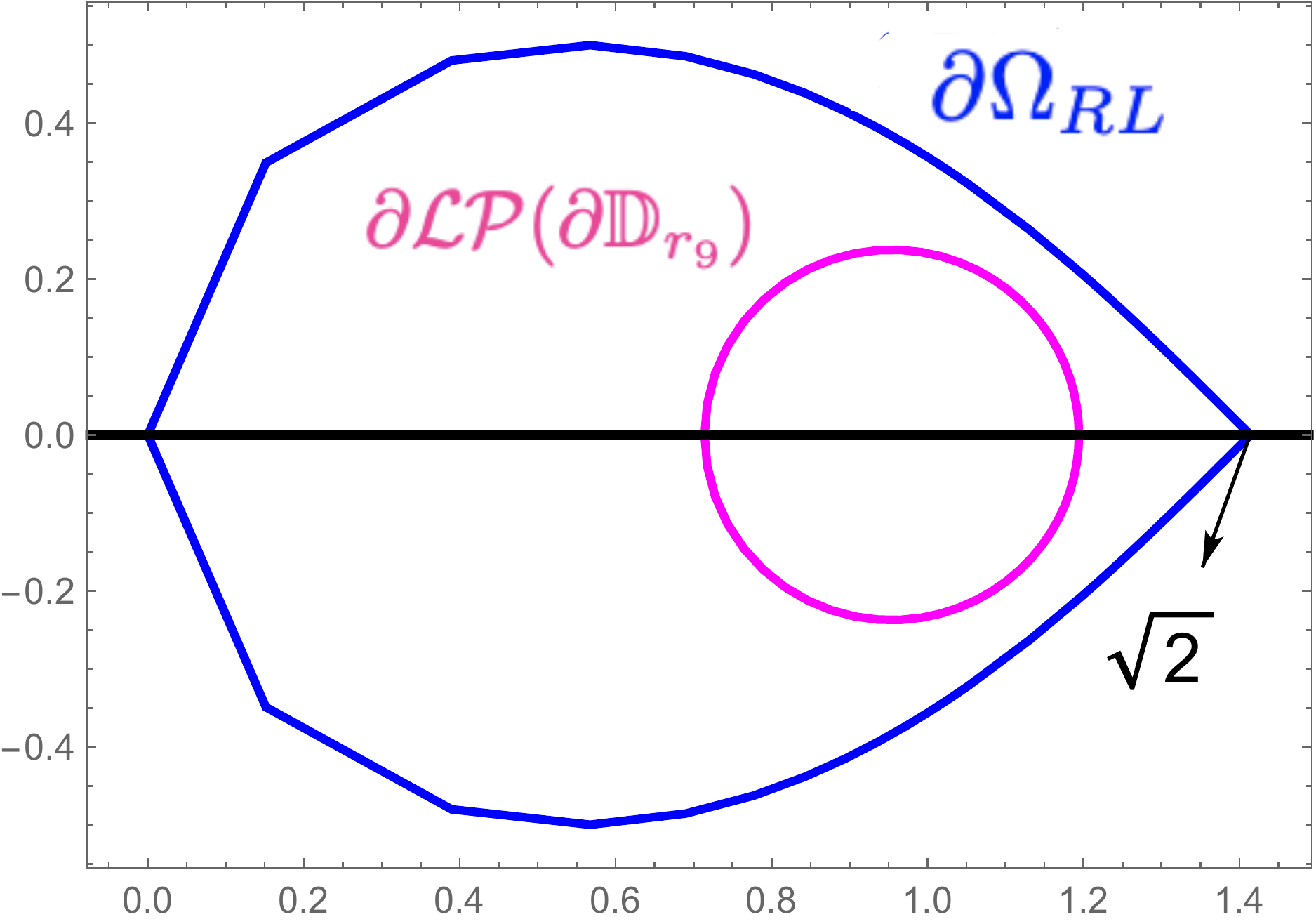}} 
\caption{Above figures correspond to Corollary \ref{EQN415}, with radii $r_{8}$ and $r_{9}$ given by: 
(a) $r_{8}\approx0.376\ldots$ (b) $r_{9}\approx0.283\ldots.$}
 \label{fig:6}
\end{figure}

\noindent We now define the class $\mathfrak{F}$ constructed with the help of ratios of two analytic functions $f,g\in\mathcal{A},$ studied by Mundalia and Kumar \cite{Mundalia n Sivaprasad(2022)}.
\[\mathfrak{F}=\left\{f\in\mathcal{A}:\operatorname{Re}\left(\frac{f(z)}{g(z)}\right)>0 \text{ } \& \text{ } \operatorname{Re}\left(\frac{(1-z)^{1+A}g(z)}{z}\right)>0, -1 \leq A \leq 1\right\}.\] In patricular, for  $A=-1$ and $A=1,$ the class $\mathfrak{F}$ reduces to the following classes, respectively,
\[\mathfrak{F}_{1}=\left\{f\in\mathcal{A}:\operatorname{Re}\left(\frac{f(z)}{g(z)}\right)>0 \text{ } \& \text{ } \operatorname{Re}\left(\frac{g(z)}{z}\right)>0\right\}\]
and 
\[\mathfrak{F}_{2}=\left\{f\in\mathcal{A}:\operatorname{Re}\left(\frac{f(z)}{g(z)}\right)>0 \text{ } \& \text{ } \operatorname{Re}\left(\frac{(1-z)^{2}g(z)}{z}\right)>0\right\}.\]
\noindent For proving  the next theorem, we require the following lemma of Ravichandran et al. \cite{V.Ravi}. 
\begin{lemma} \label{l11} 
If $p\in\mathcal{P}_{n}[A,B],$ then for $|z|=r$
\[\left|p(z)-\frac{1-A B r^{2n}}{1-B^{2}r^{2n}}\right| \leq \frac{|A-B|r^{n}}{1-B^{2}r^{2n}}.\]
Particularly, if $p\in\mathcal{P}_{n}(\alpha),$ then 
\[\left|p(z)-\frac{1+(1-2 \alpha)r^{2n}}{1-r^{2n}}\right|\leq \frac{2(1-\alpha)r^{n}}{1-r^{2n}}.\]	
\end{lemma}
\begin{theorem}
Let $-1\leq A\leq 1,$ and suppose  $f\in\mathcal{F}_{\mathcal{LP}},$ then the sharp $\mathfrak{F}-$radius is given by
\[\mathcal{R}_{\mathfrak{F}}(\mathcal{F}_{\mathcal{LP}})=\frac{1}{2 A+3}\left(\sqrt{A^2+12 A+28}-(5+A)\right)=:R_{\mathfrak{F}}.\]
\end{theorem}
\begin{proof}
Since $f\in\mathfrak{F},$ then by definition of the class $\mathfrak{F},$ we have $f(z)=p_{1}(z)g(z)$ and $g(z)=zp_{2}(z)(1-z)^{-(1+A)},$ where for each $i=1,2,$ $p_{i}:\mathbb{D}\to\mathbb{C}$ are analytic functions such that $p_{i}(0)=1$ and $\operatorname{Re}p_{i}(z)>0.$ This leads to  $f(z)=zp_{1}(z)p_{2}(z)(1-z)^{-(1+A)},$ and as a consequence of logarithmic differentiation, we obtain
\[\frac{zf'(z)}{f(z)}=\frac{1+Az}{1-z}+\frac{zp_{1}'(z)}{p_{1}(z)}+\frac{zp_{2}'(z)}{p_{2}(z)}.\]
For each $-1\leq A\leq 1$,  Lemma \ref{l11} leads to, 
\begin{equation}\label{EQN411}
    \left|\frac{zf'(z)}{f(z)}-\frac{1+Ar^{2}}{1-r^{2}}\right|\leq \frac{(5+A)r}{1-r^{2}}=R.
\end{equation}
Further, for each  $|z|=r\leq R_{\mathfrak{F}},$ one can observe that 
\begin{equation}\label{EQN412}
    \frac{1}{2}\leq 1 \leq a =\frac{1+Ar^{2}}{1-r^{2}}\leq \frac{1+A R_{\mathfrak{F}}^{2}}{1-R_{\mathfrak{F}}^{2}}<\frac{3}{2}.
\end{equation}
Infact inequalities \eqref{EQN411} and \eqref{EQN412} yields the inequality,
\[\frac{(5+A)r}{1-r^{2}}\leq \frac{3}{2}-\frac{1+Ar^{2}}{1-r^{2}},\]
provided $r\leq R_{\mathfrak{F}}.$ 
Due to Lemma \ref{EQN406}, it is clear that the disc $|u-a|<R$  lies in $\Omega_{\mathcal{LP}}.$ Further, at $z_{0}=R_{\mathfrak{F}}$ the  function $f_{\mathfrak{F}}(z)$ defined as $f_{\mathfrak{F}}(z)=z(1+z)^2/(1-z)^{3+A}$ acts as the extremal function.  
\end{proof}
\begin{corollary}
Let $f\in\mathcal{F}_{\mathcal{LP}},$ then sharp $\mathfrak{F}_{1}-$ radius and $\mathfrak{F}_{2}-$ radius for the class $\mathcal{F}_{\mathcal{LP}}$ are respectively given as 
\begin{enumerate}[(i)]
    \item $\mathcal{R}_{\mathfrak{F}_{1}}(\mathcal{F}_{\mathcal{LP}})=\sqrt{17}-4 \approx 0.123...$
    \item  $\mathcal{R}_{\mathfrak{F}_{2}}(\mathcal{F}_{\mathcal{LP}})=(\sqrt{41}-6)/5\approx 0.080...$
\end{enumerate}
\end{corollary}

\begin{theorem}
Let $\delta=(\pi  \sqrt{\beta -1}/\sqrt{2}),$ where $1<\beta<3/2,$ and suppose $f\in\mathcal{F}_{\mathcal{LP}},$ then $\mathcal{M}(\beta)-$ radius is 
$r_{\beta}=1+2 \left(\cot \delta\right)^{2}-2 |\sec \delta / (\tan^{2} \delta)|.$
\end{theorem}
\begin{proof}
From Lemma \ref{L41}, it can be viewed that
\begin{align*}
    \operatorname{Re}{\mathcal{LP}}(z)\leq \mathcal{LP}(-r)&=1-\frac{2}{\pi^{2}}\left( \log \left(\frac{1+i \sqrt{r}}{1-i \sqrt{r}}\right)\right)^{2}= 1 + \frac{2}{\pi^{2}} \left(\tan ^{-1}\left(\frac{2 \sqrt{r}}{1-r}\right)\right)^{2}.
    \end{align*}
As $f\in\mathcal{F}_{\mathcal{LP}},$ then assume that $zf'(z)/f(z)=p(z).$ Due to the above inequality $\operatorname{Re}p(z)\leq \mathcal{LP}(-r).$ Moreover, $\mathcal{LP}(-r)\leq\beta$ provided $r\leq r_{\beta},$ where  $r_{\beta}$ is the root of the equation $(1-\beta)\pi^{2}+2(\tan^{-1}({2\sqrt{r}/(1-r)}))^{2}=0$ for $1<\beta<3/2.$ Equality here occurs for the function $f_{0}\in\mathcal{A},$ given by \eqref{EQN42}.
\end{proof}

If $f(z)$ and $g(z)$  be analytic functions in $|z|<r,$ then $f(z)$ is said to be majorized by $g(z),$ denoted as $f(z)\ll g(z),$ in $|z|<r,$ if $|f(z)|\leq|g(z)|$ in $|z|<r.$ Equivalently, a function $f(z)$ is said to be be majorized by $g(z),$ if there exists an analytic $\Psi(z)$ with $|\Psi(z)|\leq 1$ in $\mathbb{D}$ and $f(z)=\Psi(z)f(z)$ for all $z\in\mathbb{D}.$ For recent update on majorization for starlike and convex function, see \cite{Gangania & kumar(2021),Tang n Dang(2019)}. In the next theorem, we determine sharp majorization radius for the class $\mathcal{F}_{\mathcal{LP}}.$

\begin{theorem}
Let $f\in\mathcal{A}$ and  suppose that $g\in\mathcal{F}_{\mathcal{LP}}.$ Further assume that $f(z)$ is majorized by $g(z)$ in $\mathbb{D},$ i.e $f(z) \ll g(z),$ then for $|z|\leq r_{m}\approx 0.4220\dots,$
\[|f'(z)|\leq |g'(z)|,\]
where $r_{m}$ is the unique positive root of the following equation
\begin{align}\label{EQN403}
    2 \pi ^2 r-(1-r^2) (\pi ^2-2 (\log ((1+\sqrt{r})/(1-\sqrt{r})))^{2})=0.
\end{align}
\end{theorem}
\vskip -1cm
\begin{proof}
Suppose $0\leq r< r^{*}=\tanh^2(\pi /2 \sqrt{2})\approx 0.646\ldots,$ then due to Remark \ref{EQN49} we conclude that, $g\in\mathcal{F}_{\mathcal{LP}}$ qualifies to be a Ma-Minda type function in $|z|< r^{*}.$
Further let $w(z)$ be a Schwarz function in $\mathbb{D}$ with $w(0)=0,$ then by definition of subordination,
\[\frac{zg'(z)}{g(z)} = \mathcal{LP}(w(z)).\]
Note that for each $|z|=r<1,$ the inequality $|\mathcal{LP}(w(z))|\leq |\mathcal{LP}(r)|$ holds. Now for $|z|=r < r^{*},$ we obtain
\begin{align}\label{EQN400}
    \left|\frac{g(z)}{g'(z)}\right| = \frac{|z|}{|\mathcal{LP}(z)|} \leq \frac{r}{1-|\mathcal{P}_{0}(r)|}=\frac{r}{\mathcal{LP}(r)}.
\end{align}
As $f(z)$ is majorized by $g(z)$ in $\mathbb{D},$ we find from the definition of majorization, \[f(z)=\psi(z)g(z).\] Upon differentiating the above equality and suitable rearrangement of terms, we obtain
\begin{align}\label{EQN401}
    f'(z)=g'(z)\left(\psi'(z)\frac{g(z)}{g'(z)}+\psi(z)\right).
\end{align} Additionally, as a result of the Schwarzian inequality, we have
\begin{align}\label{EQN402}
    |\psi'(z)|\leq \frac{1-|\psi(z)|^{2}}{1-|z|^{2}}, \quad (z\in\mathbb{D}).
\end{align}
Moreover, from equations \eqref{EQN400}-\eqref{EQN402}, we deduce
\[|f'(z)|\leq \left(|\psi(z)|+\frac{r(1-|\psi(z)|^{2})}{(1-r^{2})(\mathcal{LP}(r))}\right)|g'(z)|.\]
Substituting $|\psi(z)|=\sigma$ $(0\leq \sigma \leq 1),$ results in 
\[|f'(z)|\leq \Psi(r,\sigma)|g'(z)|,\] where 
\[\Psi(r,\sigma)=\sigma+\frac{r(1-\sigma^{2})}{(1-r^{2})(\mathcal{LP}(r))}.\]
We need to determine $r_{m}\leq r^{*}$ so that 
\[r_{m}=\max\{r\in[0,r^{*}]:\Psi(r,\sigma)\leq 1 \text{ } \forall \sigma \in [0,1]\}.\]
Equivalently if $\Phi(r,\sigma):=(1-r^2) (\mathcal{LP}(r))-r(1+\sigma)$  then we need to determine \[r_{m}=\max\{r\in[0,r^{*}]:\Phi(r,\sigma)\geq 0\text{ } \forall \sigma \in [0,1]\}.\]  Since $\partial \Phi/\partial \sigma=-r<0,$ then $\max_{\sigma\in [0,1]}{\Phi(r,\sigma)}=\Phi(r,0)=:\phi_{0}(r).$ Further it is evident that, as $\phi_{0}(0)=1>0$ and $\phi_{0}(r^{*})=-r^{*}<0,$ then there exists $r_{m}\leq r^{*},$ a smallest positive root of the equation given in \eqref{EQN403} such  that $\phi_{0}(r)\geq 0$ for each $r\in[0,r_{m}].$ This completes the proof.
\end{proof}
\noindent In 2017, Peng and Zhong \cite{Peng n Zhong(2017)}, introduced the class $\Omega \subset \mathcal{A},$ defined as 
\[\Omega=\left\{f\in\mathcal{A}:|zf'(z)-f(z)|<1/2\right\}.\] We conclude this section by determining sharp $\Omega-$radius for the class $\mathcal{F}_{\mathcal{LP}}.$ 

\begin{theorem}
Let $f\in \mathcal{F}_{\mathcal{LP}},$ then $f\in\Omega$ in $|z|<r_{\mathcal{L}}\approx 0.522\ldots$ is the smallest positive root of \[4f_{0}(r)(\log((1+\sqrt{r})/(1-\sqrt{r})))^{2}= \pi^{2}\] and 
\begin{align}\label{EQN417}
    g_{0}(z) = z  \left( \exp \int_{0}^{z} \frac{\mathcal{P}_{0}(-t)}{t}dt \right)&= z +\frac{8}{\pi ^2}z^{2}-\frac{8}{3\pi^{4}} (\pi ^2-12)z^3  +\frac{8}{135 \pi ^6}(1440\nonumber \\& \quad-360 \pi ^2+23 \pi ^4) z^4-\cdots .
\end{align}This is a sharp estimate.
\end{theorem}
\vskip -1cm
\begin{proof}
Since $f\in\mathcal{F}_{\mathcal{LP}},$ then as a consequence of Remark \ref{EQN404} for $|z|=r<1,$ we have
\[\left|\frac{zf'(z)}{f(z)}-1\right|<|\mathcal{LP}(r)-1|=|\mathcal{P}_{0}(r)|.\]
Due to the growth theorem as mentioned in \cite [Theorem 1]{Gangania n Kumar(2021)Trans} 
and Theorem \ref{EQN416}, we observe that $|f(z)|\leq g_{0}(r),$ where $g_{0}(r)$ is given by \eqref{EQN417}. Further 
\[|zf'(z)-f(z)|= |f(z)|\left|\frac{zf'(z)}{f(z)}-1\right|\leq g_{0}(r)|\mathcal{P}_{0}(r)|.\] Thus $g_{0}(r)|\mathcal{P}_{0}(r)|\leq 1/2$ provided $|z|<r_{\mathcal{L}}\approx 0.522864.$ Hence the result is established. 
\end{proof}

\subsection{Sufficient Conditions for the class $\mathcal{F}_{\mathcal{LP}}$}
In this section, we establish some sufficient conditions for the class $\mathcal{F}_{\mathcal{LP}}.$ Below, we state a Lemma given by Jack \cite{Jack(1971)}, which is utilised in context of the class under consideration.
\begin{lemma}\cite[Lemma 1, p.470] {Jack(1971)} \label{L42}
Let $\nu(z)$ be a non-constant analytic function in $\mathbb{D},$ such that $\nu(0)=0.$ If $|\nu(z)|$ attains its maximum value on the circle $|z|=r$ at a point $z_{0},$ then $z_{0}\nu'(z_{0})= k \nu(z_{0}),$ where $k$ is real and $k\geq 1.$ 
\end{lemma}
\begin{theorem}\label{Thm41}
Suppose $0\leq t\leq 1$ and let $f\in\mathcal{A}$ satisfy the following differential inequality
\begin{align}\label{EQN405}
   \left| t\left(1+\frac{zf''(z)}{f'(z)}\right)+(1-t)\frac{zf'(z)}{f(z)}-1\right|<\frac{1}{6}(3+2 t), \quad z\in\mathbb{D},
\end{align}
then $f\in\mathcal{F}_{\mathcal{LP}}.$
\end{theorem}
\begin{proof}
Consider an analytic function $\nu(z)$ with $\nu(0)=0.$ Assume $f\in\mathcal{A}$ such that
\[\frac{zf'(z)}{f(z)} -1=\frac{1}{2}\nu(z).\] We show that $|\nu(z)|<1$ in $\mathbb{D}.$ Suppose on the contrary $|\nu(z)|\geq 1,$ then by an application of Lemma \ref{L42}, there exists $z_{0}\in\mathbb{D}$ such that  for $k\geq 1,$ $|\nu(z_{0})|=1$ and $z_{0}\nu '(z_{0}) = k \nu(z_{0}).$ Substituting  $\nu(z_{0})=e^{i\mu},$ $-\pi<\mu \leq \pi$ leads to 
\begin{align*}
   & \left|t\left(1+\frac{zf''(z)}{f'(z)}\right)+(1-t)\frac{zf'(z)}{f(z)}-1\right|
   \\ & \quad =\left|t\left(1+\frac{\nu(z_{0})}{2}+\frac{k \nu(z_{0})}{2+\nu(z_{0})}\right)+(1-t)\left(1+\frac{\nu(z_{0})}{2}\right)-1\right|
   \\ & \quad =\left|\frac{t k e^{i\mu}}{2 + e^{i\mu}}+\frac{e^{i\mu}}{2}\right| \geq \frac{1}{6}(3+2t).
\end{align*}
This is contradiction to the assumption given in \eqref{EQN405}. Thus $|\nu(z)|<1,$ which means that $zf'(z)/f(z)$ lies in the disc
$|(zf'(z)/f(z))-1|<1/2.$ Hence in view of Lemma \ref{EQN406}, (with $a=1$) required result is achieved. 
\end{proof}
For $t=1/2,$ $t=0$ and $t=1$ in Theorem \ref{Thm41}, we obtain the following corollary,
\begin{corollary}
Let $f\in\mathcal{A}$ satisfy the following differential inequalities 
\begin{enumerate}[(i)]
\item $\left|(zf'(z)/f(z)+zf''(z)/f'(z))-1\right|<4/3,$  or
\item $\left|(zf'(z)/f(z))-1\right|<1/2,$ or
\item  $\left|zf''(z)/f'(z)\right|<5/6,$ 
\end{enumerate}
 then $f\in\mathcal{F}_{\mathcal{LP}}.$
\end{corollary}

\section*{Conclusion}
In the present investigation, we introduce a class of analytic functions associated with certain parabolic region. In particular, we have considered a case when parabola is lying majorly in the left half plane and symmetric about real axis. The other cases namely, oblique parabolic regions are still open, for similar investigations. 

\section*{Appendix}
The classes we come across in the present investigation are listed below for the ready reference of the reader.

\begin{table}[h!]
\renewcommand{\arraystretch}{1.25}
 \caption{Special subclasses of Ma-Minda starlike functions for specific choices of $\phi(z)$} 
  \centering 
  \begin{tabular}{llll} 
  \hline 
    {\bf{Class $\mathcal{S}^{*}(\phi)$}}  & \textbf{$\phi(z)$} & $\phi(\mathbb{D})$ & \textbf{References}      \\ [0.8ex] 
    \hline 
    $\mathcal{S}^{*}_{\alpha,e}$     &   $\alpha+(1-\alpha)e^{z}$ & $\Omega_{\alpha,e}$ &  \cite{Khatter Siva n Ravi(2019)}  Khatter et al. \\
    $\mathcal{SL}^{*}(\alpha)$    &    $\alpha+(1-\alpha)\sqrt{1+z}$ & $\Omega_{\alpha,L}$  &  \cite{Khatter Siva n Ravi(2019)}  Khatter et al.  \\ 
    $\mathcal{S}^{*}_{\wp}$   &   $1+z e^z$   &  $\Omega_{\wp}$ &  \cite{Kumar n Kamaljeet (2021)} Kumar et al.   \\
     $\mathcal{S}^{*}_{SG}$     &   $2/(1+e^{-z})$ & $\Omega_{SG}$   &   \cite{Goel n Sivaprasad(2020)}   Goel et al.  \\  
     $\mathcal{S}_{s}$       &     $1+\sin z$    &     $\Omega_{s}$    & \cite{Cho n Virender(2019)} Cho et al. \\  
     $\mathcal{S}^{*}_{\rho}$    &   $1+\sinh^{-1}z$  & $\Omega_{\rho}$ &  \cite{Arora n Siva(2022)}    Arora et al.  \\ 
      $\mathcal{S}^{*}_{\varrho}$   &   $\cosh\sqrt{z}$  & $\Omega_{\varrho}$ & \cite{Mundalia n Sivaprasad(2022)}  Mundalia et al. \\ 
       $\Delta^*$                &  $z+\sqrt{1+z^2}$  & $\Omega_{\Delta}$ &   \cite{Raina n Sokol(2015)} Raina et al. \\ $\mathcal{S}^{*}_{\mathcal{L}}$       &   $\sqrt{1+z}$ & $\Omega_{L}$ &   \cite{Sokol(1996)}  Sok\'{o}\l  \text{ }et al.   \\ 
       $\mathcal{S}^{*}(A,B)$   &  $(1+Az)/(1+Bz)$  & $\Omega_{A,B}$ &\cite{Janowski}  Janowski  \\ 
      $\mathcal{S}^{*}(N_{e})$   &  $1+z-z^{3}/3$ & $\Omega_{N_{e}}$ &\cite{Swaminathan n Wani(2021)}   Wani et al.  \\  
       $\mathcal{S}^{*}_{p}$
      &  $1+(2/\pi^{2})(\log((1+\sqrt{z})/(1-\sqrt{z})))^{2}$ & $\Omega_{p}$ & \cite{Ronning(UCV 1993)}  Ronning \\  $\mathcal{S}^{*}_{\mathcal{RL}}$  &  $\sqrt{2}-\left(\sqrt{2}-1\right) \sqrt{(1-z)(1+2 \left(\sqrt{2}-1\right) z)}$ & $\Omega_{RL}$ &  \cite{Mendiratta n Nagpal(2014)}   Mendiratta et al. \\ \hline 
\end{tabular}
\label{EQNTable2}
\end{table}

\end{document}